\documentclass[11pt]{amsart}
\usepackage{amsxtra,amssymb,amsmath,amscd,url,listings,stmaryrd}
\usepackage[alphabetic,initials]{amsrefs}
\usepackage{cases}
\usepackage{mathrsfs}
\usepackage{bbm}
\usepackage{amssymb}
\usepackage{amscd}
\usepackage{amsfonts,latexsym,amsthm,amsxtra,mathdots,amssymb,latexsym}
\usepackage[all,cmtip]{xy}
\RequirePackage{amsmath} \RequirePackage{amssymb}
\usepackage{color}
\usepackage{colordvi}
\usepackage{multicol}
\usepackage{hyperref}
\usepackage{mathtools}
\usepackage[margin=1in]{geometry}
\usepackage{xcolor}
\hypersetup{
    colorlinks,
    linkcolor={red!50!black},
    citecolor={blue!50!black},
    urlcolor={blue!80!black}
}

\newtheorem{conjecture}{Conjecture}[section]
\newtheorem{definition}{Definition}
\newtheorem{theorem}{Theorem}[section]
\newtheorem{proposition}{Proposition}
\newtheorem*{proposition*}{Proposition}
\newtheorem{lemma}{Lemma}
\newtheorem{corollary}{Corollary}
\newtheorem{Remark}{Remark}[section]

\newcommand{\cc}{{\mathbf Y}}
\newcommand{\M}{{\mathscr{M}}}
\newcommand{\kk}{N}
\newcommand{\N}{{\mathbb N}}
\newcommand{\PP}{\mathscr{P}}
\newcommand{\rr}{r}
\def \i {{\rm i}}
\def \d {{\rm d}}

\newcommand{\J}{{\mathscr{U}}}

\newcommand{\E}{{\mathbb E}}
\newcommand{\X}{{\mathbb X}}

\newcommand{\R}{{\mathbb R}}
\newcommand{\gl}{Y}

\newcommand{\cons}{\mathbf a}

\newcommand{\B}{{\mathbf{B}}}
\newcommand{\al}{{\mathbf c({\ell})}}
\newcommand{\cl}{{\mathbf c_{\ell}}}

\newcommand{\dl}{{\mathbf  D_{\ell} }}
\newcommand{\dd}{\spadesuit}
\begin{document}
\title{A note on log-type GCD sums and derivatives of the Riemann zeta function}
 \author[Daodao Yang]{Daodao Yang}

\address{Institute of Analysis and Number Theory \\ Graz University of Technology \\ Kopernikusgasse 24/II, 
A-8010 Graz \\ Austria}

\email{yang@tugraz.at \quad yangdao2@126.com}

\maketitle

\begin{abstract}
    In \cite{Deri1}, we defined  so-called ``log-type" GCD sums and proved the lower bounds $\Gamma^{(\ell)}_1(N) \gg_{\ell} \left(\log\log N\right)^{2+2\ell}$. We will establish the  upper bounds $\Gamma^{(\ell)}_1(N)\ll_{\ell} \left(\log \log N\right)^{2+2\ell}$ in this note,  which generalizes G\'{a}l's theorem on GCD sums (corresponding to the case $\ell = 0$).  This result will be proved by two different methods. 
The first method is unconditional. We establish sharp upper bounds for spectral norms along $\alpha-$lines when $\alpha$ tends to $1$ with certain fast rates. As a corollary, we obtain upper bounds for  log-type GCD sums. The second method is conditional. We prove that  lower bounds for  log-type GCD sums $\Gamma^{(\ell)}_1(N)$ can produce lower bounds for large values of derivatives of the Riemann zeta function on the 1-line. So from conditional upper bound for $\left| \zeta^{(\ell)}\left(1+\i t\right)\right|$, we obtain  upper bounds for   log-type GCD sums.   %which is the correct order. 
\end{abstract}

 \section{Introduction}

The main goal of this note is to establish the following result on log-type GCD sums.
\begin{theorem}\label{lgcd}
Fix $\ell \in [0, \infty)$. For $N \geqslant 100$, we have\footnote{Here $(m, n)$ denotes the greatest common divisor of $m$ and $n$, and $[m, n]$ denotes the least common multiple of $m$ and $n$.} 
	\begin{align}\label{lGCD}
	N (\log \log N)^{2 + 2\ell}	\ll_{\ell} \sup_{|\M| = N} \sum_{m, n\in \M} \frac{(m,n)}{[m,n]}\log^{\ell} \Big(\frac{m}{(m,n)}\Big)\log^{\ell}\Big(\frac{n}{(m,n)}\Big) \ll_{\ell} N (\log \log N)^{2 + 2\ell}\,,
	\end{align}
	where the supremum is taken over all subsets   $\M \subset \mathbb N$ with size $N$.
\end{theorem}
Theorem \ref{lgcd}  generalizes G\'{a}l's theorem \cite{G} on GCD sums (see \eqref{Gal}), which corresponding to the case $\ell = 0$ in Theorem \ref{lgcd}.   Thus we only need to prove the theorem when  $\ell \in (0, \infty)$.  

%In order to establish the upper bounds in \eqref{lGCD}, I prove the following upper bounds by using the value distributions of  random zeta function via the Lewko--Radziwi\l\l~~ method \cite{LR}.  

Let $\ell \in [0, \infty)$ and $\sigma \in \R$ be given,  define the normalized log-type GCD sums $\Gamma_{\sigma}^{(\ell)}(N)$ as we did in  \cite{Deri1}:
\begin{equation*}  
\Gamma_{\sigma}^{(\ell)}(N):\,= \sup_{|\M| = N} \frac{1}{N}\sum_{m, n\in \M} \frac{(m,n)^{\sigma}}{[m,n]^{\sigma}}\log^{\ell} \Big(\frac{m}{(m,n)}\Big)\log^{\ell}\Big(\frac{n}{(m,n)}\Big)\,.
\end{equation*}
%where the supremum is taken over all subsets   $\M \subset \mathbb N$ with size $N$. In the above definition, $(m, n)$ denotes the greatest common divisor of $m$ and $n$ and $[m, n]$ denotes the least common multiple of $m$ and $n$.
In \cite{Deri1}, we  proved that $\Gamma^{(\ell)}_1(N) \gg_{\ell} \left(\log\log N\right)^{2+2\ell}$,  for $\forall \ell \in (0, \infty)$. In this note, we will establish the corresponding upper bounds  $\Gamma^{(\ell)}_1(N) \ll_{\ell} \left(\log\log N\right)^{2+2\ell}$ by two different methods. One method is unconditional and for $\forall \ell \in (0, \infty)$, while another method is conditional  and only for  $\forall \ell \in \N.$

We have the following new result for spectral norms, which is a key ingredient for the unconditional method. 
\begin{theorem}\label{thm:spec} Let $A \in (0, \infty)$ be fixed.
Let $\mathbf{c} = (c(1), c(2), \ldots, c(n), \ldots,) \in \mathbb C^{\N}$ and  let $\M \subset \N$. Then, for $\alpha = \alpha(\kk) = 1 - \frac{A}{ \log \left(\log \kk \cdot \log_3 \kk \right)}$\,, we have
\begin{equation*} 
\sup_{\substack{ |\M| = \kk\\ \\ \|\mathbf{c} \|_2 = 1}}\, \sum_{n, m \in \M}
 \frac{(n, m)^{\alpha}}{[n, m]^{\alpha}}  c(n) \overline{c(m)}
\leqslant \left( \frac{\exp\left(2\gamma + 2e^A - 2 \right)}{\zeta(2)}   + o(1) \right)
\cdot (\log\log \kk)^2, \quad as \,\, \kk \to \infty.
\end{equation*}
\end{theorem}

\begin{Remark}
 We write $\log_j$
for the $j$-th iterated logarithm, so
for example, $\,\log_2 N :\,= \log\log N,$\, $\log_3 N :\,= \log\log\log N$.
\end{Remark}

The constant $\frac{\exp\left(2\gamma + 2e^A - 2 \right)}{\zeta(2)} $ in Theorem \ref{thm:spec} is sharp since we have the following corresponding lower bounds on GCD sums.
\begin{theorem}\label{thm:gcd} Let $A \in (0, \infty)$ be fixed. Let $\M \subset \N$. Then, for $\alpha = \alpha(\kk) = 1 - \frac{A}{ \log \log \kk }$\,, we have
\begin{equation*} 
\sup_{\substack{ |\M| = \kk}}\, \sum_{n, m \in \M}
 \frac{(n, m)^{\alpha}}{[n, m]^{\alpha}}  
\geqslant \left( \frac{\exp\left(2\gamma + 2e^A - 2 \right)}{\zeta(2)}   + o(1) \right)
\cdot \kk(\log\log \kk)^2, \quad as \,\, \kk \to \infty.
\end{equation*}
\end{theorem}

An application of  Theorem \ref{thm:spec} is to prove the following results concerning upper bounds for a modified version of log-type GCD sums.

\begin{corollary}\label{UncondUpper}
Let $\ell \in (0, \infty)$ be fixed. Let $\M \subset \N$.   Then
\begin{equation*}
  \sup_{|\M| = \kk}\, \frac{1}{\kk} \sum_{n, m \in \M} \frac{(n, m)}{[n, m]} \log^{2\ell} \left(\frac{[n, m]}{(n, m)}\right) \leqslant \left( \cons_{\ell} + o(1) \right)\cdot  (\log\log \kk)^{2+2\ell}, \quad as \,\, \kk \to \infty.
\end{equation*}
where~~ $\cons_{\ell}$ is the positive constant defined by
\begin{equation*}
 \cons_{\ell}:\, = \min_{A > 0}   \frac{\exp\left(2\gamma + 2e^{2\ell A} - 2 \right)}{\zeta(2)A^{2\ell}} \,.
\end{equation*}
\end{corollary}

By the above corollary, we have the following result.

\begin{corollary}\label{UncondLogGCD}
Let $\ell \in (0, \infty)$ be fixed.  Then
\begin{equation}\label{UnCond_Bd}
  \Gamma_1^{(\ell)} (\kk) \leqslant \left(4^{-\ell} \cdot \cons_{\ell} + o(1) \right)\cdot  (\log\log \kk)^{2+2\ell}, \quad as \,\, \kk \to \infty,
\end{equation}
where~~ $\cons_{\ell}$ is the positive constant defined as in Corollary \ref{UncondUpper}.
\end{corollary}
%\begin{remark}
    %Let $A_{\ell}$ be the solution to the equation $2A_{\ell}\cdot e^{2\ell A_{\ell}} = 1?$, then we have $A_{\ell} = \frac{1}{2\ell} \left(\log \ell - \log_2 \ell + O\left( \frac{\log_2 \ell}{\log \ell} \right)\right)  ?? $. Thus $\log \left(4^{-\ell} \cdot \cons_{\ell}\right) = -\ell \log 4 + 2 e^{2 \ell A_{\ell}} -2\ell \log A_{\ell} + O(1)=\left( 2 + o(1)\right) \ell \log \ell  $, as $\ell \to \infty$. On the other hand, numerical computations give  $4^{-1}  \cons_{1} \approx 27.6$, $4^{-2} \cons_{2} \approx 861.5$, and $4^{-3}\cons_{3} \approx 43087$.
%\end{remark}
\begin{Remark}
   Asymptotically,  $\log \left(4^{-\ell} \cdot \cons_{\ell}\right) \sim  2 \ell \log \ell  $, as $\ell \to \infty$. On the other hand, numerical computations give  $4^{-1}  \cons_{1} \approx 27.6$, $4^{-2} \cons_{2} \approx 861.5$, and $4^{-3}\cons_{3} \approx 43087$.
\end{Remark}

%\begin{Remark}
%The above Corollary  generalizes G\'{a}l's theorem on GCD sums, which corresponding to the case $\ell = 0$ (see \eqref{Gal}).   
%\end{Remark}

In \cite{Deri1}, the motivation of the study of such log-type GCD sums was to produce large values of $	\left| \zeta^{(\ell)}\left(1+\i t\right)\right| $. Our second method comes from the explicit connections between log-type GCD sums and  $	\left| \zeta^{(\ell)}\left(1+\i t\right)\right| $, which is presented in the following Proposition \ref{MainOne}.  We will prove Proposition \ref{MainOne}  using the resonance  methods (see \cites{V, So, Hi, CA, BS1, BSNote, BS2, delaBT}). 

\begin{proposition}\label{MainOne}
Fix $\epsilon > 0,\,\, \beta \in [0, 1), \,\, \kappa \in (0,  1-\beta)$ and $\ell \in \N$. 

Assume that $~\cl$ is  a positive constant for which there exists an infinite sequence of positive integers $~N_1 < N_2 < \cdots < N_n <\cdots $ such that  \begin{equation}\label{asump}
\Gamma^{(\ell)}_1 \left( N_n \right) \geqslant \cl \,\left(\log\log N_n \right)^{2+2\ell}\,, \quad \quad \forall n \in \N.
\end{equation}
Then for all sufficiently large  $n\in \N$,  we can find a real number $t $ with $~N_n^{\beta} \leqslant t^{\kappa} \leqslant N_n ~,$ such that \begin{align*}
 \left|\zeta^{(\ell)}\left(1+it\right)\right| \geqslant \left(\sqrt {\cl\cdot \dd} - \epsilon\right) \left(\log\log t \right)^{\ell+1}\,, \end{align*}
 where $~\dd$ is the positive constant defined as in Theorem \ref{UpB}.
\end{proposition}

\begin{Remark}
In \cite{Deri1}, we mentioned that we could use  log-type GCD sums to establish lower bounds for the maximum of $~\vert\zeta^{(\ell)} \left(1+it\right) \vert$, but without giving such a proof. Instead, in \cite{Deri1} we  used a different proposition to establish lower bounds for the maximum of $~\vert\zeta^{(\ell)} \left(1+it\right) \vert$ on the shorter interval $[\frac{T}{2}, T]$.
\end{Remark}

%Corollary ??? gives the following conditional upper bounds for log-type GCD sums.

The following Theorem is a corollary of  Proposition \ref{MainOne}.
\begin{theorem}\label{UpB}
 Fix $\epsilon > 0$ and $\ell \in \mathbb{N}$.  For all sufficiently large $N \in \N$, we have 
\begin{equation}\label{Cond_Bd}
\Gamma^{(\ell)}_1(N) \leqslant \left(\,\frac{\dl}{\dd}    + \epsilon \, \right)\, \left(\log \log N\right)^{2+2\ell}\,,
\end{equation}
where $~\dl$ and $\dd$ are  defined by
\begin{align*}
 \dl:\;= \limsup_{t \to \infty} \left|\frac{\zeta^{(\ell)}(1+it)}{(\log \log t)^{\ell+1}}\right|  ^2 
,\quad \quad \quad \quad
    \dd:\,= \max_{ 0 <  x \leqslant 2 } \frac{e^{-x}}{1+ 2\sum_{n = 0}^{\infty} e^{-xn^2} }\,.
\end{align*}
%In particular, when assuming RH, we will have $~\dl < \infty$  and thus obtain the conditional upper bound $~\Gamma^{(\ell)}_1(N) \ll_{\ell}  \left(\log\log N\right)^{2+2\ell}$.

\end{theorem}

Furthermore, we have the following conditional result on  the Riemann Hypothesis (RH).
\begin{proposition} \label{Main}
Assume RH. Fix  $\ell \in \mathbb{N}$.   For large  $t \in \R$, we have $\left|\zeta^{(\ell)}\left(1+it\right)\right| \ll_{\ell}\;\left(\log \log t\right)^{\ell+1}.$
\end{proposition}

So when assuming RH, we will have $~\dl < \infty$,  and thus obtain the conditional upper bound $~\Gamma^{(\ell)}_1(N) \ll_{\ell}  \left(\log\log N\right)^{2+2\ell}$. On the other hand, if we assume the following unproven conjecture (which is similar to a conjecture of Granville and  Soundararajan on character sums \cite{GS01}) , then we will have $\max_{T\leqslant t\leqslant 2T}\left|\zeta^{(\ell)}\left(1+it\right)\right| \sim \mathbf{Y}_{\ell} \left(\log_2 T \right)^{\ell+1},\,\text{as}\,~ T \to \infty,\,$  by \cite{DeriL}. Here, $ \mathbf{Y}_{\ell} = \int_0^{\infty} u^{\ell} \rho (u) du$ and $\rho(u)$ denotes  the Dickman function.  

\begin{conjecture}[\cite{DeriL}]\label{DYC}
There exists a constant $A>0$ such that
 for any
$1\leqslant x \leqslant T$, $2 T \leqslant t \leqslant 5T$, we have, uniformly,  
$$
\sum_{n\leqslant x}  \frac{1}{n^{it}} = \sum_{\substack{n \leqslant x \\ P(n) \leqslant y}}  \frac{1}{n^{it}}   + o\left( \Psi(x,y)
\right)\,,  \quad \text{as}\quad T \to \infty\,,
$$
where $y= (\log T +\log^2 x) (\log \log T)^A$. Here, $P(n)$ denotes  the largest prime factor of $n$ and $\Psi(x,y)$ denotes the number of integers smaller than $x$ with $P(n) \leqslant y$.
\end{conjecture}

\begin{Remark}
   By the conditional asymptotic formula for $\max_{T\leqslant t\leqslant 2T}\left|\zeta^{(\ell)}\left(1+it\right)\right|$, we will have $\dl = \mathbf{Y}_{\ell}^2$ when assuming  Conjecture \ref{DYC}.  And $\log \mathbf{Y}_{\ell} \sim \ell \log \ell$ \cite{DeriL}, so conditionally  we    obtain  $\log \dl\sim  2 \ell \log \ell  $, as $\ell \to \infty$.
\end{Remark}

\begin{Remark}
We  compute that $0.14149 < \dd < 0.14151$, by using the following inequality
\begin{align*}
  \frac{e^{-x}}{1+ 2\sum_{n = 0}^{3} e^{-xn^2} + \frac{2 e^{-16x}}{1- e^{-9x}} }  <  \frac{e^{-x}}{1+ 2\sum_{n = 0}^{\infty} e^{-xn^2} } <   \frac{e^{-x}}{1+ 2\sum_{n = 0}^{5} e^{-xn^2} }\,, \quad x > 0\,. \end{align*}
  And one can compute that \,$\cc_{1} = e^{\gamma}$,  $\cc_{2} = 3e^{\gamma}/2$ and $\cc_{3} = 17e^{\gamma}/6$ \cite{DeriL}.   So if we  assume Conjecture \ref{DYC}, then   $\mathbf{D}_1/\dd \approx 22. 4$,  $\mathbf{D}_2/\dd \approx 50.4$ and $\mathbf{D}_3/\dd \approx 180$. In this case, the unconditional bound \eqref{UnCond_Bd} is weaker  than the conditional bound \eqref{Cond_Bd}.
\end{Remark}
%\begin{remark}

%\begin{Remark}
%By the recent work in \cite{DeriL}, we have $~~\lim_{\ell \to \infty} \dl = \infty.$
%\end{Remark}

We will give two different proofs for Proposition  \ref{Main}. And  the implied constants are effectively  computable. In Sections \ref{1stP}, \ref{2ndP}, after the proof, we give examples of computations for the implied constants. In particular, one of methods gives   $\left|\zeta^{\prime \prime}\left(1+it\right)\right| \leqslant 20 e^{\gamma}\left(\log \log t\right)^{3}+ O\left(\left(\log \log t\right)^2\right) $ on RH. These constants are further improved in \cite{DeriL}.  The main goal is not to sharp the implied constants but to  present different approaches. In particular, one of the proof of Proposition  \ref{Main} relies on the following Proposition \ref{RHnbd}, which could be viewed as analogs of  Theorem \ref{thm:spec}. 

\begin{proposition}\label{RHnbd}
    Assume RH and let $A  > 0$ be fixed. When $t \to \infty$, we have$$
  \log |\zeta(\sigma + it)|   \leqslant\begin{cases}
			\log \big|\zeta(1 + it)\big| + e^{2 A} -1 +  o(1), \quad & \text{if} \quad 1 -  \frac{A}{\log\log t} \leqslant \sigma \leqslant 1.\\
             \log \big|\zeta(1 + it)\big| + 2 A  +  o(1), \quad & \text{if} \quad 1  \leqslant \sigma \leqslant 1 + \frac{A}{\log\log t} . 
		 \end{cases}$$
\end{proposition}
And we have the following result, which is similar to Theorem \ref{thm:gcd}.
\begin{theorem}\label{thm:limitOfzeta} Let $A \in (0, \infty)$ be fixed.  Then, for $\sigma = \sigma(T) = 1 - \frac{A}{ \log \log T }$\,, we have
\begin{equation*} 
\max_{T \leqslant t \leqslant 2T}\left| \zeta\left( \sigma + it\right)\right|\geqslant \left( \exp\left(\gamma + e^A - 1 \right)   + o(1) \right)
 \log\log T, \quad as \,\, T \to \infty.
\end{equation*}
\end{theorem}
We mention that  Zaitsev \cite{Zaitsev},  Kalmynin \cite{Kalmynin} and  Bondarenko-Seip  \cite{BSNote}, also investigated large values of $|\zeta(\sigma + it)|$, when $\sigma \to 1^{-}.$

By Proposition \ref{RHnbd}, the problem reduces to give upper bounds for $\big|\zeta(1 + it)\big| $.  Littlewood’s classical result on RH states that 
$ \big|\zeta(1 + it)\big| \leqslant \left(2 e^{\gamma}+o(1)\right) \left(\log\log t \right)$, as $t \to \infty.$ In \cite{LLS}, Lamzouri-X. Li-Soundararajan obtained the following result (on RH)
\begin{align}\label{LLS}
	\vert \zeta\left(1+it\right)\vert \leqslant 2 e^{\gamma}  \left(\log\log t - \log 2  + \frac{1}{2} + \frac{1}{\log \log t}\right)\,, \quad \forall t \geqslant 10^{10}\,.
	\end{align}

In  \cite{Deri1}, the author studied extreme values for $\left|\zeta^{(\ell)}\left(\sigma+it\right)\right|$ when $\sigma \in [\frac{1}{2} , 1].$
In this context, we also consider  conditional upper bounds for $\left|\zeta^{(\ell)}\left(\sigma+it\right)\right|$ when $\sigma \in [\frac{1}{2} , 1)$.  The following Proposition  \ref{strip} is an easy consequence of the work of Chandee-Soundararajan \cite{CSo} and Carneiro-Chandee \cite{CC}.
 \begin{proposition} \label{strip}
Assume RH. Fix  $\,\epsilon > 0$, $\ell \in \mathbb{N}$ and $\sigma_0 \in (\frac{1}{2} , 1)$.  Let $t$ be sufficiently large, then

\text{\emph{(A)}} %\quad There exists a positive constant $~\Cl$ such that
\begin{equation*} \label{eq:Half} \left|\zeta^{(\ell)}\left(\frac{1}{2}+it\right)\right| \leqslant  \exp  \left\{\left(\frac{\log 2}{2}+\epsilon\right)\,\frac{\log t }{\log \log t}\right\}\,,\end{equation*}

\text{\emph{(B)}} \quad \begin{equation*} \label{eq:srtip} \left|\zeta^{(\ell)}\left(\sigma_0+it\right)\right| \leqslant   \exp  \left\{ \left( \frac{1}{2} + \frac{2 \sigma_0 -1}{\sigma_0(1-\sigma_0)} + \epsilon \right) \, \frac{(\log t)^{2 - 2\, \sigma_0}}{\log \log t} \right\}\,. \end{equation*}

\end{proposition}

In \cite{Deri1}, when $\ell \in \N$ and $\sigma \in [\frac{1}{2} , 1)$ are given, we use GCD sums (rather than log-type GCD sums) to produce large values of $\left|\zeta^{(\ell)}\left(\sigma+it\right)\right|$ . The reason is that when $\ell$ is fixed, there is no significant difference between GCD sums and log-type GCD sums. One can easily prove the following Proposition  \ref{logGCDstrip} based on the work of de la Bret\`eche-Tenenbaum \cite{delaBT} and  Aistleitner-Berkes-Seip  \cite{ABS}. Therefore, the most interesting case for log-type GCD sum is the case when $\sigma = 1$.

\begin{proposition}\label{logGCDstrip}
Fix  $\ell \in [0, \infty)$ and $\sigma \in (\frac{1}{2} , 1)$. 
 
 \text{\emph{(A)}} As $N \to \infty$, we have
\begin{align*}
 \Gamma^{(\ell)}_{\frac{1}{2}}(N) =   \exp \Big\{\big(2\sqrt{2} +o(1)\big)\sqrt{\frac{\log N \,\log_3 N}{\log_2 N}}\Big\}\, .
\end{align*}

\text{\emph{(B)}} There exists positive constants  $c_{\sigma}$ and $C_{\sigma}$  depending on $\sigma$ such that for sufficiently large $N$, we have \begin{equation*}
\exp  \Big\{c_{\sigma}\cdot  \frac{(\log N)^{1 - \sigma}}{(\log_2 N)^{\sigma}}\Big\}\leqslant \Gamma^{(\ell)}_{\sigma}(N) \leqslant  \exp \Big\{C_{\sigma}\cdot  \frac{(\log N)^{1 - \sigma}}{(\log_2 N)^{\sigma}}\Big\}.
\end{equation*}
\end{proposition}

\vspace{0.7cm}

Let $  \sigma \in (0, 1]$ be given and let $\M \subset \N$ be a finite set. The greatest common divisors (GCD) sums $S_{\sigma}(\M)$ of $\M$ are defined as follows:
\begin{align*}
   S_{\sigma}(\M):\,=\sum_{m, n \in \mathcal{M}}\frac{(m,n)^{\sigma}}{[m, n]^{\sigma}}\,\,\, . \quad
\end{align*}

The case $\sigma = 1$ was  studied by G\'{a}l \cite{G}, who proved that
\begin{align}\label{Gal} (\log\log N)^2 \ll \sup_{|\M| = N}  \frac{ S_{1}(\M)}{|\M|} \ll (\log\log N)^2.\end{align}

G\'{a}l’s proof is a difficult
combinatorial argument and his argument  highly depends on the fact that $\frac{(m, n)}{[m, n]}$ is multiplicative. It's not clear that G\'{a}l’s combinatorial method can be used to establish the upper bounds $\Gamma^{(\ell)}_1(N) \ll_{\ell} \left(\log\log N\right)^{2+2\ell}$ since the log-type GCD sums are highly non-multiplicative.

The asymptotically sharp constant in \eqref{Gal}  was  found by Lewko and Radziwi\l\l  \,\, in \cite{LR}, where they proved that  
\begin{align*}
\left(\frac{6 e^{2\gamma}}{\pi^{2}}+o(1) \right) \left(\log\log  N\right)^2 \leqslant  \sup_{|\M| = N}  \frac{ S_{1}(\M)}{|\M|} \leqslant \left(\frac{6 e^{2\gamma}}{\pi^{2}}+o(1) \right) \left(\log\log  N\right)^2.
\end{align*}

Given $\sigma \in (\frac{1}{2}, 1),$ Aistleitner, Berkes, and Seip \cite{ABS}
 proved the following  result for GCD sums $S_{\sigma}(\M)$, where  $c_{\sigma}$ and $C_{\sigma}$ are positive constants only depending on $\sigma:$
\begin{align}\label{gcd}
\emph{\emph{exp}}  \Big\{c_{\sigma}\cdot  \frac{(\log N)^{1 - \sigma}}{(\log_2 N)^{\sigma}}\Big\} \leqslant  \sup_{|\M| = N}  \frac{ S_{\sigma}(\M)}{|\M|} \leqslant \emph{\emph{exp}}  \Big\{C_{\sigma}\cdot  \frac{(\log N)^{1 - \sigma}}{(\log_2 N)^{\sigma}}\Big\}.
\end{align}

For $ \sigma = \frac{1}{2}$, based on constructions of \cites{BS, BS1}, de la Bret\`eche and Tenenbaum \cite{delaBT} proved  the following result, improving early results of Bondarenko-Seip\cites{BS, BS1}. 
\begin{align}\label{GCD: 1/2}
\sup_{|\M| = N}  \frac{ S_{\frac{1}{2}}(\M)}{|\M|} = \emph{\emph{exp}}  \Big\{\big(2\sqrt{2} +o(1)\big)\sqrt{\frac{\log N \,\log_3 N}{\log_2 N}}\Big\}, \quad \text{as}~~ N \to \infty.
\end{align}

By theorems of Aistleitner-Berkes-Seip and G\'{a}l on GCD sums, one can find that $\sigma = 1$ is a transition point for the GCD sums. Our Theorem \ref{thm:spec} show that as long as $\sigma \to 1^{-}$ with sufficiently fast rates, the optimal GCD sums on the $\sigma$-line will have the same size as the  optimal GCD sums on the $1$-line. So Theorem \ref{thm:spec} could be interesting in this point of view. And by Proposition \ref{RHnbd} and Theorem \ref{thm:limitOfzeta}, we see that this phenomenon also happens for the   Riemann zeta function (when assuming RH), i.e., the maximal size of  $\left|\zeta(\sigma + it)\right|$ has similar behavior as GCD sums, when  $\sigma \to 1^{-}$ with sufficiently fast rates.

\section{The random  zeta-function $\zeta(s, \X)$ and expectation estimates}
%\subsection{The random zeta-function $\zeta(s, \X)$}
%\subsection{Distributional estimates}
Let $\{\X(p)\}_p$ be a sequence of independent random variables (one for each prime $p$),  uniformly distributed
on the unit circle $\{z \in \mathbb{C}:\, |z| = 1 \}$. For an integer $n$, we let
$$
\X(n) :\,= \prod_{p^{\alpha} \| n} \X(p)^{\alpha}.
$$
The random zeta-function $\zeta(s, \X)$  is defined as the following:
$$
\zeta(s, \X) :\,= \prod_{p} \Big (1 - \frac{\X(p)}{p^{s}} \Big )^{-1} = \sum_{n = 1}^{\infty} \frac{\X(n)}{n^{s}}.
$$
The product and series both converge almost surely when   $\Re(s) > \tfrac 12$ (for instance, see \cite[page 4 and 6]{Sound}). 
Note that
$$
\mathbb{E}[\X(n)\overline{\X(m)}] = \begin{cases}
1 & \text{ if } n = m \\
0 & \text{ if } n \neq m 
\end{cases}.
$$
See \cites{Youness, LR, Sound} for more information and  applications of $\zeta(s, \X)$. 
Here, we only need estimates on the expectation of large powers of $|\zeta(s, \X)|$, which is related to upper bounds for GCD sums. The  lemma below is proved by  Lewko-Radziwi\l\l \cite{LR},   using ideas  from Lamzouri \cite[Lemma 2.1]{Youness}, who proved upper bounds for $\log \mathbb{E} [|\zeta(\alpha, \X)|^{2\gl}]$ for fixed $\alpha \in (\frac{1}{2}, 1). $

\begin{lemma}[Lemma 6 of \cite{LR}]\label{LzIMRN}
We have the following bound,
$$
\log \mathbb{E} [|\zeta(1, \X)|^{2\gl}]
\leqslant
2\gl(\log\log \gl+ \gamma ) + O\left(\frac{\gl}{\log \gl}\right) .
$$
\end{lemma}
%We will prove that if $\alpha$ tends to 1 with a sufficient fast speed, then 

 \section{Preliminary Results on  the Riemann zeta function $\zeta$}

  \subsection{Lemmas on RH}
 
 In this subsection, we collect several conditional upper bounds for the Riemann zeta function, which will be used in later sections. 
 \begin{lemma}[Littlewood, \,Thm 13.13\,\cite{M2}]\label{Mont}
 Assume RH. Then
 \begin{align*}
\left| \frac{\zeta^{'}}{\zeta}(\sigma +it)\right| \leqslant \sum_{ n \leqslant (\log t)^2} \frac{\Lambda(n)}{n^{\sigma}}+ O\left( (\log t)^{2 - 2\sigma} \right) \,,
\end{align*}
uniformly for $~~\frac{1}{2} + \frac{1}{\log\log t} \leqslant \sigma \leqslant \frac{3}{2},\,\, |t| \geqslant 10.$
 \end{lemma}

\begin{lemma}[Chandee,  Soundararajan\,\,\cite{CSo}]\label{CS}
Assume RH. For large real numbers $t$, we have
\begin{equation*} \label{eq:CS} \left|\zeta\left(\frac{1}{2}+it\right)\right| \leqslant  \exp  \left\{\frac{\log 2}{2}\,\frac{\log t }{\log \log t}  +O\left(\frac{\log t \, \log \log \log t }{(\log \log t)^2} \right)\right\}. \end{equation*}

\end{lemma}

 \begin{lemma}[Carneiro,  Chandee \,\cite{CC}]\label{CC}
Let $\alpha = \alpha(t)$ be a real-valued function with $\frac{1}{2} < \alpha \leqslant 1$. Assume RH. For large real numbers $t$, we have
$$
  \log |\zeta(\alpha + it)|   \leqslant\begin{cases}
			\log \left( 1 +  (\log t)^{1- 2\alpha}\right) \frac{\log t}{2\log \log t} + O\left( \frac{(\log t)^{2-2\alpha}}{(\log \log t)^2} \right) , & \text{if\, $ (\alpha - \frac{1}{2})\log \log t = O(1);$}\\
			\log(\log \log t) + O(1), & \text{if\, $(1- \alpha)\log \log t = O(1);$ }\\
            \left( \frac{1}{2} + \frac{2\alpha - 1}{\alpha (1 - \alpha)} \right)\frac{(\log t)^{2 - 2\alpha}}{\log \log t}  + \log (2\log \log t) + O\left(\frac{(\log t)^{2 - 2\alpha}}{(1-\alpha)^2(\log \log t)^2}\right), & \text{ otherwise.}
		 \end{cases}$$
\end{lemma}
 
  \subsection{Bell polynomials
and Fa\`a di Bruno's formula}
 In this subsection, we present a formula for $\frac{\zeta^{(n)}}{\zeta}\left(s   \right)$ by using Bell polynomials
and Fa\`a di Bruno's formula. We will apply this formula for $s = 1 + it$ in Section \ref{2ndP}.
\begin{definition}\cite[page 134]{Comb}
The partial  Bell polynomials $~\B_{n ,k}\left(x_1, x_2, \dots, x_{n -k +1}\right)$ are defined by
 \begin{align*}
 \B_{n ,k}\left(x_1, x_2, \dots, x_{n -k +1}\right):\, =    \sum \frac{n!}{j_1 ! j_2 ! \cdots j_{n-k+1}!} \left(\frac{x_1}{1!}\right)^{j_1}\left(\frac{x_2}{2!}\right)^{j_2}\cdots \left(\frac{x_{n-k+1}}{(n-k+1)!}\right)^{j_{n-k+1}},
 \end{align*}
 where the summation takes place over all sequences $j_1, j_2, j_3, \dots, j_{n-k+1}$ of non-negative integers such that the following two conditions are satisfied:
 \begin{align}
    &j_1+ j_2 + j_3 + \cdots + j_{n-k+1} = k, \label{eq:B1}\\
    & j_1+ 2 j_2 + 3 j_3 + \cdots + (n-k+1) j_{n-k+1} = n. \label{eq:B2}
 \end{align}
 The n-th complete exponential Bell  polynomial $~\B_n$  is defined by the following sum:
 \begin{align*}
     \B_n(x_1, x_2, \dots, x_n):\, = \sum_{k=1}^{n}\B_{n ,k}\left(x_1, x_2, \dots, x_{n -k +1}\right)\,.
 \end{align*}
\end{definition} 
 
In particular, we can compute $\B_1(x_1) = x_1$,  $\B_2(x_1, x_2) = x_1^2 + x_2$ and  $\B_3(x_1, x_2, x_3) = x_1^3 + 3 x_1 x_2 + x_3$.
 
  Fa\`a di Bruno's formula \cite[page 137]{Comb} is a chain rule   on higher derivatives, which can be expressed in terms of Bell polynomials as follows:
  \begin{align*}
      \frac{d^n}{dx^n} f\left(g\left(x\right)\right) = \sum_{k=1}^{n}f^{(k)}(g(x)) \B_{n, k} \left(g^{\prime}(x), g^{\prime\prime}(x), \cdots, g^{(n-k+1)}(x)  \right)\,.
  \end{align*}
 
Applying Fa\`a di Bruno's formula to $f(s) = e^s$ and 
$g(s) = \log \zeta(s)$ ( \cite[page 19]{PRZZ}), one can get 
\begin{align}\label{FaaDi}
 \nonumber &\zeta^{(n)}(s) =  \sum_{k=1}^{n}\zeta(s) \, \B_{n, k} \left(\frac{\zeta^{\prime}}{\zeta}(s),\, \frac{d}{ds}\frac{\zeta^{\prime}}{\zeta}(s),\, \cdots, \frac{d^{n-k}}{ds^{n-k}}\frac{\zeta^{\prime}}{\zeta}(s)\right) \,,  \\   &\frac{\zeta^{(n)}}{\zeta}\left(s   \right) =  \sum_{k=1}^{n} \, \B_{n, k} \left(\frac{\zeta^{\prime}}{\zeta}(s),\, \frac{d}{ds}\frac{\zeta^{\prime}}{\zeta}(s),\, \cdots, \frac{d^{n-k}}{ds^{n-k}}\frac{\zeta^{\prime}}{\zeta}(s)\right) = \B_{n} \left(\frac{\zeta^{\prime}}{\zeta}(s),\, \frac{d}{ds}\frac{\zeta^{\prime}}{\zeta}(s),\, \cdots, \frac{d^{n-1}}{ds^{n-1}}\frac{\zeta^{\prime}}{\zeta}(s)\right) \,.   
\end{align}

\section{Proof of Theorem \ref{thm:spec} and  Corollary \ref{UncondUpper} and \ref{UncondLogGCD}}

\subsection{Proof of Theorem \ref{thm:spec} }
We will use the method of Lewko-Radziwi\l\l \cite{LR}  to prove the Theorem. 
By \cite[page 287-288]{LR}, if $|\M| = \kk$ and $\| \mathbf{c} \|_2^2 =  \sum_{n = 1}^{\infty}\left|c(n)\right|^2 = 1$, then we have
\begin{equation}\label{LR:Ineq} 
\zeta(2\alpha) \cdot  \sum_{n, m \in \M}
 \frac{(n, m)^{\alpha}}{[n, m]^{\alpha}} \cdot c(n) \overline{c(m)}
\leqslant
e^{2V}  + \kk \cdot \mathbb{E} [|\zeta(\alpha, \X)|^{2\gl + 2}] \cdot e^{-2\gl V}, \quad \forall\, \gl, V > 0\,.
\end{equation}
So the problem now reduces to give suitable upper bounds for the expectation $\mathbb{E} [|\zeta(\alpha, \X)|^{2\gl + 2}]$, when $\alpha \to 1^{-1}$ with certain converging rates. Before stating such upper bounds, we give the following lemma, which will be helpful for us to bound  error terms when using the prime number theorem.
\begin{lemma}\label{IntegralLimiBound}Let $A \in (0, \infty)$ be fixed. We have the following bound,
\begin{align*}
Int(\alpha; A):\,= \int_2^{\exp\left(\frac{A}{1-\alpha}\right)} \frac{1}{t^{\alpha}}\exp\left(-\sqrt {\log t} \right) dt \ll_A  1  , \quad \forall \alpha \in [\frac{1}{ 2}, 1) \bigcap [1-\frac{A}{\log 2}, \infty)\,.
\end{align*}
\end{lemma}
\begin{proof}
Note that $\exp\left(-\sqrt {\log t} \right) \ll 1/\left( \log t \right)^2\,$ and $\frac{d}{dt}\left(-1/ \log t    \right) = 1/\left(t \log^2 t \right)  $. So
$$Int(\alpha; A) \ll \int_2^{\exp\left(\frac{A}{1-\alpha}\right)} \frac{t}{t^{\alpha}}\frac{d}{dt}\left( \frac{-1}{\log t}  \right)dt \,.$$
Now integration by parts gives
$$Int(\alpha; A) \ll  -\frac{e^A}{A}(1 - \alpha) + \frac{e^A}{\log 2} \ll_A  1 \,.$$
\end{proof} The following lemma is a key ingredient for the proof of Theorem \ref{thm:spec}.
\begin{lemma}\label{LogExpect}Let $A \in (0, \infty)$ be fixed. We have the following bound,
\begin{align*}
\log \mathbb{E} [|\zeta(\alpha, \X)|^{2\gl}]
\leqslant 2\gl(\log\log \gl+ \gamma + e^A - 1 )+ O_A\left(\frac{\gl}{\log \gl}\right)\,, \quad for \quad \alpha = 1 - \frac{A}{\log \gl}  \,.    
\end{align*}
\end{lemma}

\begin{proof}
Note that
$$
\mathbb{E}[|\zeta(\alpha, \X)|^{2\gl}] = \prod_{p} E_\gl(p, \alpha)
\text{ with }
E_\gl(p, \alpha) = \mathbb{E} \Big [ \Big | \Big ( 1 - \frac{\X(p)}{p^{\alpha}}
\Big )^{-2\gl} \Big | \Big ].
$$
 The crucial point is that the following inequality is valid for all real  $\theta$ and all positive  $\alpha \leqslant 1$
$$ \frac{\left(1 - \frac{e^{i\theta}}{p}\right)\left(1 - \frac{e^{-i\theta}}{p}\right)}{\left(1 - \frac{e^{i\theta}}{p^{\alpha}}\right)\left(1 - \frac{e^{-i\theta}}{p^{\alpha}}\right)} \leqslant \left( \frac{1 - \frac{1}{p}}{1-\frac{1}{p^{\alpha}}} \right)^2\,.$$
By the above observation and the inequality $\int_{-\pi}^{ \pi} f(\theta) d\theta /\int_{-\pi}^{ \pi} g(\theta) d\theta \leqslant \sup_{\theta \in [-\pi, \pi]} f(\theta) / g(\theta)$ (which holds for any two positive functions $f$ and $g$), we have
$$ \frac{E_\gl(p, \alpha)}{E_\gl(p, 1)} \leqslant \left( \frac{1 - \frac{1}{p}}{1-\frac{1}{p^{\alpha}}} \right)^{2\gl}\,. $$

As a result, we obtain
$$ \log \frac{E_\gl(p, \alpha)}{E_\gl(p, 1)} \leqslant 2\gl\log \left( \frac{1 - \frac{1}{p}}{1-\frac{1}{p^{\alpha}}} \right) \leqslant 2\gl (1-\alpha) \left(  \frac{\log p}{p^{\alpha}} +  \frac{\log p}{(p^{\alpha}-1)p^{\alpha}} \right)  \,. $$

 When $\gl/ p^{2 \alpha} \ll 1$, we have (also see \cite[Lemma 4]{GS})
$$
E_\gl(p, \alpha) = \frac{1}{2\pi} \int_{-\pi}^{\pi} \Big (1 - \frac{e^{i\theta}}{ p^{\alpha}} \Big )^{-\gl} \cdot
\Big (1 - \frac{e^{-i\theta}}{p^{\alpha}} \Big )^{-\gl} d\theta = I_0\left(\frac{2\gl}{p^{\alpha}} \right )
\left(1 + O\left(   \frac{\gl}{p^{2\alpha}}\right )\right )\,,
$$
where $I_0(t)$ is the 0-th modified Bessel function defined as
$$I_0(t):\, = \frac{1}{\pi} \int_{0}^{\pi} e ^{t \cos \theta} d\theta = \sum_{n = 0}^{\infty} (t/2)^{2n} / (n!)^2, \quad \forall t \in \R\,.$$ From the expansion of $I_0(t)$, we have
$0 < I_0^{\prime}(t) \ll t$ for $0 < t \leqslant 2$.
Thus when $p  \geqslant \gl^{1/\alpha}$, we get
$$ \log  I_0\left(\frac{2\gl}{p^{\alpha}}\right) -  \log  I_0\left(\frac{2\gl}{p}\right) = \int_{\frac{2\gl}{p}}^{\frac{2\gl}{p^{\alpha}}} \frac{I_0^{\prime}(t)}{I_0(t)}dt \ll \int_{\frac{2\gl}{p}}^{\frac{2\gl}{p^{\alpha}}} t dt \ll \frac{\gl}{p^{\alpha}}\left(\frac{\gl}{p^{\alpha}} - \frac{\gl}{p}\right)\ll \gl^2 (1-\alpha)\frac{\log p}{p^{2\alpha}}\,. $$

Combining these bounds gives 
\begin{align*}
\log\frac{ \E[|\zeta(\alpha, \X)|^{2\gl}]}{ \E[|\zeta(1, \X)|^{2\gl}]} \leqslant &2\gl (1-\alpha)   \sum_{p < \gl^{1/\alpha}} \frac{\log p}{p^{\alpha}} + 2\gl (1-\alpha)\sum_{p < \gl^{1/\alpha}} \frac{\log p}{(p^{\alpha}-1)p^{\alpha}}\\& + O(1) \cdot \sum_{p \geqslant \gl^{1/\alpha}} \gl^2 (1-\alpha)\frac{\log p}{p^{2\alpha}} +  O(1) \cdot \sum_{p \geqslant \gl^{1/\alpha}}
\frac{\gl}{p^{2\alpha}}.
\end{align*}

By the prime number theorem and Lemma \ref{IntegralLimiBound}, the first term is bounded by
$$ \leqslant 2\gl (e^A - 1) + O_A\left(\frac{\gl}{\log \gl}\right) \,,$$

and the other three terms are bounded by 
$$ \ll_A \frac{\gl}{\log \gl}\,.$$

The proof now follows from Lemma \ref{LzIMRN}.
\end{proof}

Now in the  inequality \eqref{LR:Ineq},  we let
$$
V = \log_3 \kk + \gamma + e^A - 1 +   \frac{2}{\log_3 \kk} \text{ and }
\gl = \log_3 \kk \cdot \log \kk\,,
$$
then 
$$
\alpha = 1 - \frac{A}{\log \gl}\,.
$$
With the choice of $\gl$ and $V$, by Lemma \ref{LogExpect} we have
$\mathbb{E} [|\zeta(\alpha, X)|^{2\gl + 2} ] \cdot e^{-2\gl V} \ll \kk^{-1}\,.$ Clearly,  $\zeta(2\alpha) \to \zeta(2)$ when $\kk \to \infty$. The claim of the theorem  follows immediately.

\subsection{Proof of   Corollary \ref{UncondUpper}}

By the inequality $\log X \leqslant \frac{1}{\epsilon} X ^{\epsilon}, \,(\forall \epsilon > 0, \forall X \geqslant 1)$, we obtain
\begin{equation*}
 \sum_{n, m \in \M} \frac{(n, m)}{[n, m]} \log^{2\ell} \left(\frac{[n, m]}{(n, m)}\right) \leqslant \left(\frac{1}{\epsilon}\right)^{2\ell}  \sum_{n, m \in \M}  \frac{(n, m)^{1- 2 \ell \epsilon}}{[n, m]^{1- 2 \ell \epsilon}}\,.
\end{equation*}
Let $|\M| = \kk $. Let $c(n) = 1/\sqrt \kk $ if $n \in \M$, and  $c(n) = 0$ if $n \notin \M$. Take $\epsilon = A/\log \left(\log \kk \cdot \log_3 \kk \right)$\,, where $A$ is a positive number to be chosen later. Then  by Theorem \ref{thm:spec}, we have
\begin{equation*} 
\frac{1}{\kk} \sum_{n, m \in \M}
 \frac{(n, m)^{1- 2 \ell \epsilon}}{[n, m]^{1- 2 \ell \epsilon}}  
\leqslant \Big ( \frac{1}{\zeta(2)} \exp\left(2\gamma + 2e^{2\ell A} - 2 \right)  + o(1) \Big )
\cdot (\log\log \kk)^2, \quad as \,\, \kk \to \infty.
\end{equation*}
By our choice of $\epsilon$, we have
$$ \left(\frac{1}{\epsilon}\right)^{2\ell}   = \left(\frac{1}{A^{2\ell}} + o(1) \right) (\log \log \kk )^{2 \ell}\,.$$
Combining the above two inequalities and choosing $A$ to minimize the  expression 
$$  \frac{\exp\left(2\gamma + 2e^{2\ell A} - 2 \right)}{\zeta(2)A^{2\ell}} \,, $$
we are done.

\subsection{Proof of   Corollary \ref{UncondLogGCD}}
Note that $$ \log  \frac{m}{(m, n)}  + \log \frac{n}{(m, n)} =  \log \frac{m}{(m, n)} \cdot \frac{n}{(m, n)} = \log \frac{[m, n]}{(m, n)}.$$
By the inequality $ab \leqslant (\frac{a+b}{2})^2$, we  obtain
\begin{equation*}
    \sum_{m, n\in \M} \frac{(m,n)}{[m,n]}\log^{\ell} \Big(\frac{m}{(m,n)}\Big)\log^{\ell}\Big(\frac{n}{(m,n)}\Big) \leqslant  \frac{1}{4^{\ell}} \sum_{n, m \in \M} \frac{(n, m)}{[n, m]} \log^{2\ell} \left(\frac{[n, m]}{(n, m)}\right) \,.
\end{equation*}
Now  Corollary \ref{UncondLogGCD} immediately follows from Corollary \ref{UncondUpper}.
\section{Proof of Theorem \ref{thm:gcd}}

Let $\PP(r, b) = p_1^{b - 1} \cdot \ldots \cdot p_{\rr}^{b - 1}$\,,
where $p_n$ denotes the $n$-th prime. Define $\M$ to be the set of divisors of $\PP(r, b)$, then $|\M| = b^{\rr}.$
 Then we have the following  G\'{a}l's identity\footnote{It was stated for $\alpha = 1$ in \cite{G}.}  \cite{G},
\begin{align}\label{general: ga}
    \sum_{m, n\in \M} \frac{(m,n)^{\alpha}}{[m,n]^{\alpha}}  = \prod_{p \leqslant p_\rr}
\left ( b + 2 \sum_{k = 1}^{b - 1} \frac{b - k}{p^{k\alpha}}
\right).
\end{align}

To prove the above identity, we write $ m = p_1^{\nu_1} p_2^{\nu_2} \cdots p_\rr^{\nu_\rr} $, $ n = p_1^{\beta_1} p_2^{\beta_2} \cdots p_\rr^{\beta_\rr}$   and compute the left hand side of \eqref{general: ga} as follows
\begin{align*}
   \sum_{m, n\in \M} \frac{(m,n)^{\alpha}}{[m,n]^{\alpha}} & = \sum_{\beta_{\rr} = 0}^{b -1 } \sum_{\nu_{\rr} = 0}^{b -1 } \cdots \sum_{\beta_2 = 0}^{b -1 } \sum_{\nu_2 = 0}^{b -1 }\sum_{\beta_1 = 0}^{b -1 } \sum_{\nu_1 = 0}^{b -1 } \left( p_1^{-|\nu_1 - \beta_1|} p_2^{-|\nu_2 - \beta_2|}  \cdots p_{\rr}^{-|\nu_{\rr} - \beta_{\rr}|} \right)^{\alpha}\\
   & = \prod_{p \leqslant p_{\rr}} \left(  \sum_{\beta = 0}^{b -1 } \sum_{\nu = 0}^{b -1 } \left( p^{-|\nu - \beta|}  \right)^{\alpha} \right)\\
   & = \prod_{p \leqslant p_\rr}
\left ( b + 2 \sum_{k = 1}^{b - 1} \frac{b - k}{p^{k\alpha}}
\right),
   \end{align*}
which is the right hand side of \eqref{general: ga}.

 Let $\rr = [ \log \kk / \log\log \kk ]$, then $p_{\rr} \sim \log \kk$ by the prime
number theorem.  Let $b$ be the integer satisfying that
$$
b^{\rr} \leqslant \kk < (b + 1)^{\rr},
$$
then  $b^{\rr} \sim \kk$, as $\kk \rightarrow \infty$.
 Choose any set $\M' \subset \N$  such that $\M \subset \M'$ and $|\M'| = \kk.$  Then the GCD sum over $\M'$  is at least as large as the GCD sum over $\M$.
 
 Following   Lewko-Radziwi\l\l  \,\,in \cite{LR}, we use G\'{a}l's identity for the GCD sum and split the product into several parts:
\begin{align}\label{GCD: Gal}
\nonumber \sum_{m, n\in \M} \frac{(m,n)^{\alpha}}{[m,n]^{\alpha}}  & = b^r \prod_{p \leqslant p_\rr} \Big ( 1 + 2 \sum_{k = 1}^{b - 1} \frac{1}{p^{k\alpha}}
\cdot \Big ( 1 - \frac{k}{b} \Big ) \Big ) \\ & \geqslant (1 + o(1)) \kk
\prod_{p \leqslant p_\rr} \frac{\left(1 - \frac{1}{p} \right)^{2}}{\left(1 - \frac{1}{p^{\alpha}} \right)^{2}} 
\times \prod_{p \leqslant p_\rr} \left(1 - \frac{1}{p} \right)^{-2}
\times
\prod_{p \leqslant p_\rr} \left(1 + 2 \sum_{k = 1}^{b - 1}
\frac{1}{p^{k\alpha}} \cdot\left ( 1 - \frac{k}{b} \right)
\right) \left (1 - \frac{1}{p^{\alpha}} \right)^2\,.
\end{align}

By Mertens' third theorem, the second product is asymptotically
equal to $(e^{\gamma} \log p_{\rr})^2 \sim (e^\gamma \log \log \kk)^2$
as $\kk \rightarrow \infty$. And when $\kk \rightarrow \infty$, the  last  product converges 
to
$$
\prod_{p} \Big (1 + 2 \sum_{k = 1}^{\infty} \frac{1}{p^{k}}
\Big ) \Big ( 1 - \frac{1}{p} \Big )^2 = \frac{1}{\zeta(2)}.
$$

So it remains to prove that the first product converges to  $\exp(2 e^A -2) $, which follows from the following Lemma \ref{ProdPrimeRatio},  since  $p_{\rr} \sim \log \kk$, when $N \to \infty.$

\begin{lemma}\label{ProdPrimeRatio}
 Fix $A  > 0 $ and let $\alpha = \alpha(N)  = 1 - \frac{A}{\log \log N}$. Assume that $\log X \sim \log \log N, $ as $N \to \infty.$ Then we have    
\begin{align}\label{goal}
    \sum_{p \leqslant X} \log \left(  \frac{1 - \frac{1}{p}}{1 - \frac{1}{p^{\alpha}} } \right)  = e^A -1 + o(1), \quad as \quad \kk \to \infty.
\end{align}

\end{lemma}

\begin{proof}

Let $$J_{\alpha}(p):\, = \log \left(  \frac{1 - \frac{1}{p}}{1 - \frac{1}{p^{\alpha}} } \right) = \int_{\alpha}^1 \frac{\log p}{p^x -1}\d x\,, $$
then by the integral representation of $J_{\alpha}(p)$ we have
\begin{align*}
 \frac{1}{p^{\alpha}} - \frac{1}{p} \leqslant  J_{\alpha}(p) \leqslant \frac{(1-\alpha)\log p}{p^{\alpha}} + \frac{(1-\alpha)\log p}{p^{\alpha} (p^{\alpha} - 1)} \,.
\end{align*}

By the prime number theorem and Lemma \ref{IntegralLimiBound}, we obtain
\begin{align*}
 \sum_{p \leqslant X} \frac{(1-\alpha)\log p}{p^{\alpha}} 
=  e^A -1 + o(1), \quad as \quad \kk \to \infty\,. \end{align*}

For sufficiently large $\kk$, we have
\begin{align*}
 \sum_{p \leqslant X} \frac{(1-\alpha)\log p}{p^{\alpha} (p^{\alpha} - 1)} \ll  (1-\alpha)\sum_{p \leqslant X} \frac{\log p}{p^{2\alpha} } \ll_A \frac{1}{\log \log \kk}\,. \end{align*}

Let $$\Delta_{\alpha}(p):\,=  \left( \frac{1}{p^{\alpha}} - \frac{1}{p}\right) -   \frac{(1-\alpha)\log p}{p^{\alpha}} = -\frac{1}{p^{\alpha}} \left(\frac{1}{p^{1-\alpha}} - 1 - \log \left( \frac{1}{p^{1-\alpha}} \right)  \right) \,, $$
then by Taylor expansion of $\log (1 - (1- \frac{1}{p^{1-\alpha}})) $, we obtain
\[ \left| \Delta_{\alpha}(p) \right| = -  \Delta_{\alpha}(p) \leqslant \frac{1}{2} \left( 1- \frac{1}{p^{1-\alpha}} \right)^2 p^{1- 2\alpha}\,.\]

 So when $p \leqslant X$, we have the following estimates
\[ p^{1- 2\alpha} = p^{1- 2\left(1- \frac{A}{ \log \log \kk }\right)} \leqslant \frac{1}{p} \exp\left(\log X\,\frac{2A}{\log\log \kk} \right) \leqslant \frac{1}{p}  e^{2A + o(1)}\,. \]

Uniformly for $t \geqslant 2$ and $0 \leqslant \delta \leqslant \frac{1}{2}$, we have $ 1 - \frac{1}{t^{\delta}} \leqslant \delta$, which follows from Taylor expansion of $\exp(-\delta \log 2)$. So for sufficiently large $N$, we find that
\begin{align}
  \sum_{p \leqslant X}   \left| \Delta_{\alpha}(p) \right| \leqslant \frac{1}{2}e^{2A + o(1) }\left(1 - \alpha \right)^2   \sum_{p \leqslant X} \frac{1}{p} \ll_{A} \frac{\log_3 N}{ (\log \log N)^2}\,, \
\end{align}
by Mertens' second theorem. Therefore,
\begin{align*}
    \sum_{p \leqslant X}  J_{\alpha}(p)  = \sum_{p \leqslant X} \frac{(1-\alpha)\log p}{p^{\alpha}} + o(1)
=  e^A -1 + o(1), \quad as \quad \kk \to \infty, \end{align*}
which gives \eqref{goal}.

\end{proof}

 \section{Proof of Proposition  \ref{MainOne}}

   Without loss of generality, assume that $ \kappa  + 2\epsilon <1 $.  Let $N \in \left\{\, N_1\,, N_2\,, \cdots, N_n \,,  \cdots \right\}$ and $T = N^{\frac{1}{\kappa}}\,.$ Let $\M \subset \N$ with $|\M| = N $.  We will construct a  resonator $R(t)$, following ideas from \cite{CA}, \cite{BS1} and \cite{delaBT}.
Define \[\mathcal{M}_u:\,= \Big[(1+\frac{\log T}{T})^u,(1+\frac{\log T}{T})^{u+1}\Big)\bigcap \mathcal{M} \quad (u\geqslant 0).\]
Let $\J$ be the set of integers $u$ such that $  \mathcal{M}_u \neq \emptyset$
and let $m_u$ be the minimum of $\mathcal{M}_u$  for $u\in \J$. We then set
\[ \mathcal{M}':\, = \big \{ m_u: \ u\in \J \big\} \quad
\text{and} \quad r(m_u):\,=\sqrt{ \sum_{m\in\mathcal{M}_u} 1} = \sqrt{|\mathcal{M}_u| } \] 
for $m_u \in \mathcal{M}'$. Then the  resonator $R(t)$ is defined as follows: 
\begin{equation}
   R(t):\,=\sum_{m\in \mathcal{M}'}\frac{r(m)}{m^{it}} \,.
\end{equation}

By Cauchy's inequality, one has the following  estimates \cite{delaBT}:
$$R(0)^2\leqslant   N \sum_{m\in \mathcal{M}'} r(m)^2\leqslant N  |\mathcal{M}| \,.$$

Let $A$ be a positive number. Let  $\Phi(t):\, =  \frac{1}{\sqrt{4A\pi}} e^{-\frac{1}{4A} t^2}$ with the Fourier transform $\widehat{\Phi}$ defined by   
\[ \widehat{\Phi}(\xi):\,=\int_{-\infty}^{\infty} \Phi(x) e^{-ix\xi} dx =  e^{-A \xi^2}. \]

Define the moments as follows:
\begin{align*}
 M_1(R,T):\, & = \int_{T^{\beta}}^{T} \left|R(t)\right|^2 \Phi\big(\frac{t \log T}{T} \big) dt,\\   M_2(R,T):\, & = \int_{T^{\beta}}^{T} \left| \zeta^{(\ell)}(1 + it)\right|^2 \left|R(t)\right|^2 \Phi\big(\frac{t \log T}{T} \big) dt\,,\\   \widetilde{M_2}(R,T):\, & = \int_{T^{\beta}}^{T} \left| \sum_{k \leqslant T} \frac{(\log k)^{\ell}}{k^{1+i t}}\right|^2 \left|R(t)\right|^2 \Phi\big(\frac{t \log T}{T} \big) dt\,.
\end{align*}

By the above definitions, we have 
\begin{align}\label{ratio}
    \max_{T^{\beta}\leqslant t\leqslant T}\left|\zeta^{(\ell)}\Big(1+it\Big)\right|^2 \geqslant \frac{M_2(R, T)}{M_1(R, T)}.
\end{align}

 From the proof Lemma 5 of \cite{BS2} (replacing $T$ by $  T/\log T$), one can obtain that 
\begin{align}\label{M1}
  M_1(R,T) \leqslant \int_{-\infty}^{\infty} \left|R(t)\right|^2 \Phi\big(\frac{t \log T}{T} \big) dt \leqslant  \left(1+ 2\sum_{n = 0}^{\infty} \widehat{\Phi}(n)+\epsilon\right)\, \frac{T}{\log T}  |\M| \,.
\end{align}

Define\begin{align*}
  I(R,T):\,  = \int_{-\infty}^{\infty} \left| \sum_{k \leqslant T} \frac{(\log k)^{\ell}}{k^{1+i t}}\right|^2 \left|R(t)\right|^2 \Phi\big(\frac{t \log T}{T} \big) dt.
\end{align*}
  
Then  
\begin{align}\label{IRT}
  I(R,T) \nonumber &= \sum_{j, k \leqslant T}\frac{\big(\log j\,\cdot \log k\big)^{\ell}}{j k}  \sum_{u, \nu  \in \J}r(m_u)r(m_{\nu}) \int_{-\infty}^{\infty}  \big ( \frac{jm_{\nu}}{km_u} \big)^{it} \Phi\big(\frac{t \log T}{T} \big) dt\\
  & =  \frac{T}{\log T} \sum_{j, k \leqslant T}\frac{\big(\log j\,\cdot \log k\big)^{\ell}}{j k}  \sum_{u, \nu  \in \J}r(m_u)r(m_{\nu})   \widehat{\Phi}\Big(\frac{T}{\log T} \log \frac{k m_u}{j m_{\nu}}\Big).
\end{align}  

When $j, k $ are fixed, for any $u, \nu \in \J$, one has
\begin{align}\label{diagonal}
 \sum_{\substack{ m\in \M_u,\, n \in \M_{\nu}\\ mk = n j}} 1 \leqslant \min\{|\M_u|,\, |\M_{\nu}|\}  \leqslant r(m_u)r(m_{\nu}).   
\end{align}

Note that when $m k = n j$, we have
$\frac{n \, m_u}{m \, m_{\nu}} = \frac{k m_u}{j m_{\nu}}.$
  Multiplying by $\widehat{\Phi}\Big(\frac{T}{\log T} \log \frac{k m_u}{j m_{\nu}}\Big)$ in (\ref{diagonal}) and summing index over $u, \nu \in \J$ give
  \begin{align*}
  \sum_{u, \nu  \in \J}r(m_u)r(m_{\nu}) \widehat{\Phi}\Big(\frac{T}{\log T} \log \frac{k m_u}{j m_{\nu}}\Big) \geqslant   \sum_{u, \nu  \in \J} \sum_{\substack{ m\in \M_u,\, n \in \M_{\nu}\\ mk = n j}} \widehat{\Phi}\Big(\frac{T}{\log T} \log \frac{n \, m_u}{m \, m_{\nu}}\Big) \geqslant  \sum_{\substack{ m, n\in \M\\ mk = n j}} \widehat{\Phi}(1)\,,
\end{align*}  
 where the last inequality follows from $0 \leqslant \frac{ m }{m_u}  - 1 \leqslant \frac{ \log T}{T}$ and  $0 \leqslant \frac{n }{m_{\nu} }  -1 \leqslant \frac{ \log T}{T}$.
 
 Returning to \eqref{IRT}, we obtain
 \begin{align*}
       I(R,T)  
  \geqslant  \frac{T}{\log T} \sum_{j,\, k \leqslant T}\frac{\big(\log j\,\cdot \log k\big)^{\ell}}{j k}  \sum_{\substack{ m, n\in \M\\ mk = n j}} \widehat{\Phi}(1).
 \end{align*}

When $j = \frac{ m}{(m, n)}$ and $k = \frac{n}{(m, n)} $, we have $mk= \frac{ m n}{(m, n)}   = [m, n ] = nj.$ Thus we further get the following lower bound \begin{align*}
  I(R,T)  & \geqslant \widehat{\Phi}(1)\,\frac{T}{\log T} \sum_{\substack{\frac{m}{(m, n)}, \,\frac{n}{(m,n)} \leqslant T\\ m,\, n \in \M}}\frac{\big(\log \frac{m}{(m, n)} \,\cdot \log \frac{n}{(m,n)}\big)^{\ell}}{\frac{m}{(m,n)} \frac{n}{(m, n)}} 
  \\& \geqslant \widehat{\Phi}(1)\,\frac{T}{\log T} S(\M; \ell) - E(\M; T)\,\,,
\end{align*} where $ S(\M; \ell)$ and $E(\M; T)$ are defined as follows:

\begin{align*}
   &S(\M; \ell):\, =\sum_{m, n\in \M} \frac{(m,n)}{[m,n]}\log^{\ell} \Big(\frac{m}{(m,n)}\Big)\log^{\ell}\Big(\frac{n}{(m,n)}\Big)\,\,\,\, , \quad\\
   &E(\M; T):\, = 2 \widehat{\Phi}(1)\,\frac{T}{\log T} \sum_{\substack{\frac{n}{(m, n)} > T\\ m,\, n \in \M}} \frac{(m,n)}{[m,n]}\log^{\ell} \Big(\frac{m}{(m,n)}\Big)\log^{\ell}\Big(\frac{n}{(m,n)}\Big).
\end{align*}

Using Rankin's trick, the inequality $\log X \ll_{\varepsilon} X^{\varepsilon}$ (which holds for all $X \geqslant 1$) and the upper bound in \eqref{gcd}, we bound $ E(\M; T)$ by
\begin{align*}
  E(\M; T) &
  \ll \frac{T}{\log T} \frac{1}{\sqrt[3]{T}} \sum_{\substack{ m,\, n \in \M}} \sqrt[3]{\frac{n}{(m,n)}} \frac{(m,n)}{[m,n]}\log^{\ell} \Big(\frac{m}{(m,n)}\Big)\log^{\ell}\Big(\frac{n}{(m,n)}\Big)\\& \ll \frac{T}{\log T} \frac{1}{\sqrt[3]{T}} S_{\frac{7}{12}}(\M)\\& \ll \frac{T}{\log T} \frac{1}{\sqrt[3]{T}} |\M| \exp \left\{ C_{\frac{7}{12}}\frac{(\log T)^{\frac{5}{12}}}{(\log \log T)^{\frac{7}{12}}}\right\} \ll  \frac{T}{\log T} \frac{1}{\sqrt[4]{T}} |\M|.
\end{align*}

So we have proved that 
\begin{align}\label{lowerbound}
  I(R,T)  \geqslant \widehat{\Phi}(1)\,\frac{T}{\log T} S(\M; \ell) + O\left( \frac{T}{\log T} \frac{1}{\sqrt[4]{T}} |\M|\right).
\end{align}

In the following steps, we will bound  $  \left| I(R, T) - \widetilde{M_2}(R,T) \right|$ and $  \left| M_2(R, T) - \widetilde{M_2}(R,T) \right|$. 

First, note that\begin{align*}
   \int_{|t| \leqslant T^{\beta}} \left| \sum_{n \leqslant T} \frac{(\log n)^{\ell}}{n^{1+i t}}\right|^2 \left|R(t)\right|^2 \Phi\big(\frac{t \log T}{T} \big) dt \ll R(0)^2 T^{\beta} \cdot  (\log T)^{2 \ell +2} \ll T^{\kappa + \beta} (\log T)^{2 \ell +2} |\M|\,.
\end{align*}

On the other hand, by the fast decay of $\Phi$, we find that
 \begin{align*}
   \int_{|t| \geqslant T} \left| \sum_{n \leqslant T} \frac{(\log n)^{\ell}}{n^{1+i t}}\right|^2 \left|R(t)\right|^2 \Phi\big(\frac{t \log T}{T} \big) dt \ll T^{\kappa} (\log T)^{2 \ell +2} |\M| \int_{|t| \geqslant T}  \Phi\big(\frac{t \log T}{T} \big) dt \ll o(1) |\M|.
\end{align*}

As a result,  we have
 \begin{align}\label{IMDifference}
   I(R, T) = \widetilde{M_2}(R, T) + O\Big( T^{\kappa + \beta} (\log T)^{2 \ell +2} |\M|\Big).
\end{align}

Next, let \begin{equation*}  
\quad E_1 = \frac{\ell !}{\epsilon^{\ell}}\frac{T^{\epsilon}}{t}\,.    \end{equation*}

 By  Hardy-Littlewood's  approximation formula (see \cite[Thm 4.11]{T}) for $\zeta(s)$ and Cauchy's integral formula for derivatives, we have \begin{align}\label{derivapprox}
 \zeta^{(\ell)}(1+ it) = (-1)^{\ell}\sum_{k \leqslant T} \frac{(\log k)^{\ell}}{k^{1+it}} + O\Big( E_1 \Big) \,,\quad T^{\beta} \leqslant t \leqslant T\,.
\end{align}
 From \eqref{derivapprox} one can get
\begin{align*}
 \left| \left|\zeta^{(\ell)}(1+ it)\right|^2 - \left| \sum_{k \leqslant T} \frac{(\log k)^{\ell}}{k^{1+i t}}\right|^2 \right|\ll  \left| E_1\right|^2  +  \left| \sum_{k \leqslant T} \frac{(\log k)^{\ell}}{k^{1+i t}}\right| \cdot \left| E_1\right| \,,\quad T^{\beta} \leqslant t \leqslant T.
\end{align*}

We will estimate the contributions of $E_1$  in the integrals. \begin{align*}
     \int_{T^{\beta}}^{T} \left| E_1\right|^2 \left|R(t)\right|^2 \Phi\big(\frac{t \log T}{T} \big) dt \ll T^{2\epsilon} R(0)^2  \int_{T^{\beta}}^{T} \frac{1}{t^2} \Phi\big(\frac{t \log T}{T} \big) dt \ll  T^{2\epsilon+\kappa}  |\M|,
\end{align*}
where the implied constants depend on  $\ell$ and $\epsilon$ only.

By  the Cauchy-Schwarz inequality and \eqref{M1}, 
\begin{align*}
    &\int_{T^{\beta}}^{T}\left| \sum_{k \leqslant T} \frac{(\log k)^{\ell}}{k^{1+i t}}\right| \cdot \left| E_1\right|\cdot \left|R(t)\right|^2 \Phi\big(\frac{t \log T}{T} \big) dt \\ \leqslant & \sqrt{\widetilde{M_2}(R,T)} \sqrt{\int_{T^{\beta}}^{T} \left| E_1\right|^2 \left|R(t)\right|^2 \Phi\big(\frac{t \log T}{T} \big) dt}\\
    \ll & (\log T)^{\ell +1} \sqrt{M_1(R,T)} \sqrt{ T^{2\epsilon+\kappa}  |\M|}\\
    \ll &  (\log T)^{\ell +\frac{1}{2} } T ^{\frac{1}{2} + \epsilon + \frac{\kappa}{2}} |\M|.
\end{align*}

As a result,  we obtain
\begin{align}\label{MMDifference}
   \left| M_2(R, T) - \widetilde{M_2}(R,T) \right| \ll T^{2\epsilon+\kappa}  |\M| + (\log T)^{\ell +\frac{1}{2} } T ^{\frac{1}{2} + \epsilon + \frac{\kappa}{2}} |\M|.
\end{align}

\eqref{MMDifference} together  with \eqref{lowerbound} and \eqref{IMDifference} give 
\begin{align}\label{M2}
  M_2(R,T)  \geqslant \widehat{\Phi}(1)\,\frac{T}{\log T} S(\M; \ell) +  o\left( \frac{T}{\log T}  |\M|\right) .
\end{align}

By the assumption \eqref{asump} and \eqref{ratio}, \eqref{M1}, \eqref{M2}, we obtain
\begin{align*}
    \max_{T^{\beta}\leqslant t\leqslant T}\left|\zeta^{(\ell)}\Big(1+it\Big)\right|^2 \geqslant \frac{M_2(R, T)}{M_1(R, T)}\geqslant \cl \frac{\widehat{\Phi}(1)}{1+ 2\sum_{n = 0}^{\infty} \widehat{\Phi}(n)+2\epsilon} (\log \log T)^{2\ell +2}.
\end{align*}

\section{Proof of Theorem \ref{thm:limitOfzeta}}
Following \cite{So}, let  $\Phi:\, \mathbb R \to \mathbb R$ be a smooth function, compactly supported in $[1,2]$, 
with $0 \leqslant \Phi(y) \leqslant  1$ for all $y$, and $\Phi(y)=1$ for $5/4\leqslant y\leqslant 7/4$.  Set $N = [T^{\frac{1}{2}}]$ and let $R(t):\, = \sum_{n\leqslant N}r(n)n^{-it}$. Define %Partial integration gives that
%${\hat \Phi}(y) \ll_{\nu} |y|^{-\nu}$ for any positive integer $\nu$.
\begin{align*}
    M_1(R, T) &= \int_{-\infty}^{+\infty}\left|R(t)\right|^2\Phi(\frac{t}{T})d t,\\
    M_2(R, T) &= \int_{-\infty}^{+\infty}\left(  \sum_{k \leqslant T} \frac{1}{k^{\sigma + it}} \right)\left|R(t)\right|^2\Phi(\frac{t}{T})d t.
\end{align*}
Then with a little computations (for instance, see \cites{So, BSNote}), we obtain

\begin{align*}
     M_1(R,T) &=T {\hat \Phi}(0) \left(1+O\left(T^{-1}\right)\right) \sum_{n\leqslant N} |r(n)|^2,\\
    M_2(R, T) &= T {\hat \Phi}(0)\sum_{mk = n \leqslant N} \frac{r(m)\overline{r(mk)}}{k^{\sigma}} + O \left( T^{-1}\right)\sum_{n\leqslant N} |r(n)|^2\,.
\end{align*}
 By  Hardy-Littlewood's  approximation formula, we have 
\begin{align}\label{LimitZetaRatio}
  \max_{T \leqslant t \leqslant 2T} \big|\zeta(\sigma+ it)\big|  \geqslant \frac{|M_2(R, T) |}{M_1(R,T)} + O(T^{-\sigma})\,.
\end{align}

Set $x = (\log T) / (3 \log\log T)$ and $b = [\log\log T]$.  As in \cite[page 128 -129]{BSNote}, define the function $r:\,\N \to \{0,\,1\}$ to be the characteristic function of a set $\M$, where $\M$ is  the set
of divisors of the integer $K:\, = \prod_{p \leqslant x} p^{b-1}$. Then we have
\begin{align*}
   \left|\sum_{mk = n \leqslant N} \frac{r(m)\overline{r(mk)}}{k^{\sigma}} \right| \Big/ \left( \sum_{n\leqslant N} |r(n)|^2 \right) 
    &= \prod_{p \leqslant x} \left( 1 + \sum_{k = 1}^{b-1} \left( 1 -\frac{k}{b} \right) p^{-k\sigma}   \right) \\
    & = \prod_{p \leqslant x} \left( \frac{1 - \frac{1}{p}}{1 - \frac{1}{p^{\sigma}}} \right) \prod_{p \leqslant x} \left(1 - \frac{1}{p}\right)^{-1} \prod_{p \leqslant x} \left( 1 -\frac{1}{p^{\sigma}}\right) \left( 1 + \sum_{k = 1}^{b - 1}  \left( 1 -\frac{k}{b} \right) p^{-k\sigma}    \right).
\end{align*}
By Mertens’ third theorem, the second product is asymptotically equal to $ e^{\gamma} (\log \log T)$,
as $T \to \infty$. And  the last product converges to $1$, when $T \to \infty$.  By Lemma \ref{ProdPrimeRatio}, the first product converges to $\exp (e^A - 1)$, as $T \to \infty$.  Now the theorem follows from \eqref{LimitZetaRatio}.

 \section{Proof of Proposition \ref{RHnbd} and the First Proof of Proposition  \ref{Main} }\label{1stP}
 \subsection{The proof}
 
 We will use the  identity
 \begin{align*}
 \log \big|\zeta(\sigma+ it)\big|  -\log \big|\zeta(1 + it)\big|   =  \Re\left(\log\, \zeta(\sigma+ it) - \log\, \zeta(1+ it)\right)=  \Re\left(-\int_{\sigma}^1 \frac{\zeta^{'}(\alpha+it)}{\zeta(\alpha+it)}d\alpha\right)\,.
\end{align*}

Let $A$ be a positive number. Consider the case $1 -  \frac{A}{\log\log t} \leqslant \sigma \leqslant 1$ first. 

By the estimate $ \int_{\sigma}^1 \frac{1}{n^{\alpha}}d\alpha\leqslant (1 - \sigma)\, \frac{1}{n^{\sigma}} $ and  Lemma \ref{Mont}, we obtain
 \begin{align*}
 &\bigg|\log \big|\zeta(\sigma+ it)\big|  -\log \big|\zeta(1 + it)\big|  \bigg| \\
 \leqslant &\bigg|\int_{\sigma}^1 \left| \frac{\zeta^{'}(\alpha+it)}{\zeta(\alpha+it)}\right|d\alpha\bigg|\\
 \leqslant &\bigg|\int_{\sigma}^1 \left( \sum_{ n \leqslant (\log t)^2} \frac{\Lambda(n)}{n^{\alpha}}+ O\left( (\log t)^{2 - 2\alpha} \right) \right)d\alpha\bigg|\\
 \leqslant & \left( 1- \sigma\right) \sum_{ n \leqslant (\log t)^2}  \frac{\Lambda(n)}{n^{\sigma}}    + O\left(1\right) \int_{\sigma}^1(\log t)^{2 - 2\alpha}  d\alpha
 \\\leqslant & e^{2A} -1 + O\left(\frac{1}{\log\log t}\right),
\end{align*}
 where the last inequality follows from partial summation and the prime number theorem on the Riemann Hypothesis
 $\left(\sum_{n \leqslant x}\Lambda(n) = x + O\left(\sqrt{x}\log^2 x\right)\right)$.
 
 So we have 
  \begin{align}\label{approx}
    \log \big|\zeta(\sigma+ it)\big| \leqslant \log \big|\zeta(1 + it)\big| + e^{2 A} -1 +  o(1), \quad \text{if} \quad 1 -  \frac{A}{\log\log t} \leqslant \sigma \leqslant 1.
 \end{align}
 
 For the other case $ 1 \leqslant \sigma \leqslant  1 +  \frac{A}{\log\log t}$\,,  we use the estimate $ \int_{1}^{\sigma} \frac{1}{n^{\alpha}}d\alpha\leqslant (\sigma - 1)\, \frac{1}{n} $ instead and  similarly prove that 
 \begin{align}\label{approx2}
    \log \big|\zeta(\sigma+ it)\big| \leqslant \log \big|\zeta(1 + it)\big| + 2 A  +  o(1), \quad \text{if} \quad 1  \leqslant \sigma \leqslant 1 + \frac{A}{\log\log t} .
 \end{align}
 
From \eqref{approx}  and \eqref{approx2}, we have 
\begin{align}
   \big|\zeta(\sigma+ it)\big| \leqslant \exp\left( e^{2 A} -1 +  o(1)\right) \cdot \big|\zeta(1 + it)\big| , \quad \text{if} \quad \left|\sigma -1  \right|\leqslant \frac{A}{\log\log t} .
 \end{align}

Now set $\delta = \frac{A}{\log\log t}$. Note that $\left| (\sigma+ i\widetilde{t})-(1+it)\right|\leqslant \delta$ implies $\left|\sigma -1 \right| \leqslant \delta$ and $ \left|\widetilde{t} - t \right| \leqslant \delta.$

By Cauchy's integral formula for derivatives and Littlewood's classical result on RH, we obtain that
\begin{align*}
  \left|\zeta^{(\ell)}\Big(1+it\Big)\right| &\leqslant \frac{\ell!}{\delta^{\ell}} \max_{ \substack{ |\sigma -1 | \leqslant \delta\\ |\widetilde{t} - t | \leqslant \delta}}\big|\zeta(\sigma+ i\widetilde{t})\big|\\ &\leqslant \frac{\ell!}{A^{\ell}} \left(\log\log t \right)^{\ell} \exp\left( e^{2 A} -1 +  o(1)\right) \cdot \left(2 e^{\gamma}+o(1)\right) \left(\log\log t \right)\\ &\leqslant \exp\left( e^{2 A} + \gamma -1 +  o(1)\right)\,\frac{2\ell!}{A^{\ell}}   \left(\log\log t \right)^{\ell+1}\,.
  \end{align*}

\subsection{Some Examples}
In order to optimize the constant, we let $A$  be the solution of the equation $2 e^{2x} = \ell x^{-1}$. Then numerical computations give $\left|\zeta^{\prime}\left(1+it\right)\right| \leqslant 15.2 e^{\gamma}  \left(\log \log t\right)^{2} $, $\left|\zeta^{\prime \prime}\left(1+it\right)\right| \leqslant 84.6 e^{\gamma}\left(\log \log t\right)^{3}, $  and $\left|\zeta^{(3)}\left(1+it\right)\right| \leqslant 531.5 e^{\gamma}  \left(\log \log t\right)^{4} $, for all sufficiently large $t$.

 \section{The Second Proof of  Proposition \ref{Main} }\label{2ndP}
 \subsection{The proof}\label{proof2ndP}
  Let $x = (\log t)^2,~ t \geqslant 10, ~\sigma \geqslant \sigma_0 > \frac{1}{2}.$
 
 Similar to the first formula in the proof of Lemma 2.6 of \cite{LLS}, we have
 \begin{align*}
    \sum_{n \leqslant x}\frac{\Lambda(n)}{n^{\sigma + it}}\log\left(\frac{x}{n}\right) = \frac{1}{2 \pi i} \int_{2 - i \infty}^{2 + i\infty} \left( - \frac{\zeta^{\prime}}{\zeta} \left(\sigma + it + s\right) \frac{x^s}{s^2} \right)ds.
\end{align*}
 Moving the line of integration to the left gives that
 \begin{align}\label{afterMove}
      \sum_{n \leqslant x}\frac{\Lambda(n)}{n^{\sigma + it}}\log\left(\frac{x}{n}\right) = & - \left(\log x\right) \frac{\zeta^{\prime}}{\zeta}(\sigma + it) -\left( \frac{\zeta^{\prime}}{\zeta}\right)^{\prime}(\sigma + it)\\\nonumber & - \sum_{\gamma} \frac{x^{\frac{1}{2}-\sigma+ i(\gamma - t)}}{\left(\frac{1}{2}-\sigma+ i(\gamma - t)\right)^2} - \sum_{n=1}^{\infty}\frac{x^{-2n-it-\sigma}}{(2n+it+\sigma)^2} + \frac{x^{1- \sigma - it}}{(1- \sigma - it)^2}\,,
 \end{align}
 where $\gamma$ denotes the imaginary part of nontrivial zeros of $\zeta$ (only in this subsection \ref{proof2ndP}, where it should not be confused with the Euler constant).
 
Applying $\frac{d^{\ell}}{d\sigma^{\ell}}$ on both sides of \eqref{afterMove} gives  
 \begin{align}\label{afterDeriv}
      \sum_{n \leqslant x}\frac{\left(- \log n\right)^{\ell}\Lambda(n)}{n^{\sigma + it}}\log\left(\frac{x}{n}\right) = & - \left(\log x\right) \left(  \frac{\zeta^{\prime}}{\zeta}\right)^{(\ell)}\left(\sigma + it\right) -\left( \frac{\zeta^{\prime}}{\zeta}\right)^{(\ell+1)}\left(\sigma + it\right) + \widetilde{E} ,
 \end{align}
 where by Cauchy’s integral formula for derivatives and the estimate $\sum_{\gamma} \frac{1}{b+(\gamma - t)^2} \ll \max\{\frac{1}{b},\, 1\}\, \log t$\, ($\forall b > 0$), we can bound  $ \widetilde{E}$ by 
 \[\widetilde{E} \ll \frac{\ell!}{\epsilon^{\ell}} x^{\frac{1}{2} - \sigma - \epsilon}\frac{1}{\left(\frac{1}{2} - \sigma - \epsilon\right)^2}\log t + \frac{\ell!}{\epsilon^{\ell}} x^{-2 - \sigma - \epsilon} + \frac{\ell!}{\epsilon^{\ell}}\frac{x^{1- \sigma - \epsilon}}{t^2}\,.\]
 
 Now we take $\sigma = 1$ and $ \epsilon = \frac{1}{\log\log t},$ and obtain
 \begin{align*}
      \sum_{n \leqslant x}\frac{\left(- \log n\right)^{\ell}\Lambda(n)}{n^{1 + it}}\log\left(\frac{x}{n}\right) = & - \left(\log x\right) \left(  \frac{\zeta^{\prime}}{\zeta}\right)^{(\ell)}\left(1 + it\right) -\left( \frac{\zeta^{\prime}}{\zeta}\right)^{(\ell+1)}\left(1 + it\right) + O_{\ell}\left(\left(\log\log t\right)^{\ell}\right) \,.
 \end{align*}
 
 By partial summation and the prime number theorem on RH, we have 
 \begin{align*}
   \left|   \sum_{n \leqslant x}\frac{\left(- \log n\right)^{\ell}\Lambda(n)}{n^{1 + it}}\log\left(\frac{x}{n}\right) \right| \leqslant \sum_{n \leqslant x}\frac{\left( \log n\right)^{\ell}\Lambda(n)}{n}\log\left(\frac{x}{n}\right) = \frac{\left( \log x \right)^{\ell + 2}}{\left( \ell + 1 \right)\left( \ell + 2  \right)} + O_{\ell}\left(\log x\right)\,.
 \end{align*}
 
Furthermore, by the triangle inequality, we get 
\begin{align}\label{induction}
\left|\left( \frac{\zeta^{\prime}}{\zeta}\right)^{(\ell+1)}\left(1 + it\right)\right|  \leqslant    2\left(\log \log t\right) \left|   \left(  \frac{\zeta^{\prime}}{\zeta}\right)^{(\ell)}\left(1 + it\right)\right| +  \frac{\left( 2 \log \log t \right)^{\ell + 2}}{\left( \ell + 1 \right)\left( \ell + 2  \right)} + O_{\ell}\left(\left(\log\log t\right)^{\ell}\right)\,.
\end{align}
 
 We use the convention that the 0-th derivative of a function is the function itself. Using again the estimate $\sum_{\gamma} \frac{1}{b+(\gamma - t)^2} \ll \max\{\frac{1}{b},\, 1\}\, \log t$\, ($\forall b > 0$), one can check that \eqref{induction} also holds for $\ell = 0 $. Moreover, by  Lemma \ref{Mont},
 \begin{align}\label{Initial}
\left|\left( \frac{\zeta^{\prime}}{\zeta}\right)^{(0)}\left(1 + it\right)\right|  =   \left| \frac{\zeta^{\prime}}{\zeta}\left(1 + it\right)\right|  \leqslant 2 \log \log t + O(1)\,.
\end{align}
 
 By induction and \eqref{induction}, \eqref{Initial}, we obtain
 \begin{align}\label{DerivRatio}
     \left|\left( \frac{\zeta^{\prime}}{\zeta}\right)^{(\ell)}\left(1 + it\right)\right|  \leqslant  \left( 2^{\ell+2} - \frac{2^{\ell+1}}{\ell+1}  \right)    \left(\log \log t\right)^{\ell+1} + O_{\ell}\left(\left(\log \log t\right)^{\ell}\right)\,.
 \end{align}
 
Let $\B_{\ell}$ be the  $\ell$-th complete exponential Bell polynomial, then \eqref{FaaDi} and the triangle inequality imply
 \begin{align}
     \left| \frac{\zeta^{(\ell)}}{\zeta} (1+ it) \right| \leqslant \B_{\ell}\left( \left| \frac{\zeta^{\prime}}{\zeta}\left(1 + it\right)\right|,\, \left|\left( \frac{\zeta^{\prime}}{\zeta}\right)^{\prime}\left(1 + it\right)\right|,\, \left|\left( \frac{\zeta^{\prime}}{\zeta}\right)^{(2)}\left(1 + it\right)\right|,\, \cdots , \left|\left( \frac{\zeta^{\prime}}{\zeta}\right)^{(\ell-1)}\left(1 + it\right)\right|\right)\,.
 \end{align}
 
 By  property \eqref{eq:B2} of Bell polynomials and \eqref{DerivRatio}, we get
 \begin{align}\label{ZetaLZeta}
      \left| \frac{\zeta^{(\ell)}}{\zeta} (1+ it) \right| \leqslant  \al   \left(\log \log t\right)^{\ell} + O_{\ell}\left(\left(\log \log t\right)^{\ell-1}\right)\,,
 \end{align}

where the positive constant $\al$ is defined as 
\begin{align*}
    \al: \,= \B_{\ell}\left( 2^{2} - \frac{2^{1}}{1}\,,\,  2^{3} - \frac{2^{2}}{2} \,,\, \cdots ,\, 2^{\ell+1} - \frac{2^{\ell}}{\ell}\right)\,.
\end{align*}

Thus \eqref{LLS} and \eqref{ZetaLZeta} give 
 \begin{align*}
\left| \zeta^{(\ell)} \left(1+ it\right) \right|  = \left| \zeta \left(1+ it\right) \right|  \cdot  \left| \frac{\zeta^{(\ell)}}{\zeta} (1+ it) \right|    \leqslant 2 e^{\gamma} \, \al   \left(\log \log t\right)^{\ell+1} + O_{\ell}\left(\left(\log \log t\right)^{\ell}\right)\,.
 \end{align*}

\subsection{Some Examples}
  From $\B_1(x_1) = x_1$,  $\B_2(x_1, x_2) = x_1^2 + x_2$ , $\B_3(x_1, x_2, x_3) = x_1^3 + 3 x_1 x_2 + x_3$ and 
  \[\B_1(2) = 2, \quad \B_2(2, 6) = 2^2 + 6 = 10, \quad \B_3(2, 6, \frac{40}{3}) = 2^3 + 3\times 2 \times 6 + \frac{40}{3} = \frac{172}{3},\]
  we obtain $\left|\zeta^{\prime}\left(1+it\right)\right| \leqslant 4 e^{\gamma}  \left(\log \log t\right)^{2} + O\left(\log \log t\right), $ $\left|\zeta^{\prime \prime}\left(1+it\right)\right| \leqslant 20 e^{\gamma}\left(\log \log t\right)^{3}+ O\left(\left(\log \log t\right)^2\right) , $  and $\left|\zeta^{(3)}\left(1+it\right)\right| \leqslant \frac{344}{3} e^{\gamma}  \left(\log \log t\right)^{4} + O\left(\left(\log \log t\right)^3\right)\,.$

   \section{Proof of  Proposition \ref{strip}}
 
 \subsection{Proof of  Proposition \ref{strip} (A)}

 Let $A$ be a positive number, to be chosen later.  By  Lemma \ref{CS} and  Lemma \ref{CC}, we have
 \begin{align*}
     \left|\zeta(\sigma+it)\right| \leqslant \exp\left\{\frac{\log 2}{2}   \,\frac{\log t }{\log \log t}  +O\left(\frac{\log t \, \log \log \log t }{(\log \log t)^2} \right)  \,  \right\}\,, \quad \text{if}  \quad  0 \leqslant \sigma-\frac{1}{2}  \leqslant \frac{A}{\log\log t}.
 \end{align*}

  When $\frac{1}{2} -  \frac{A}{\log\log t} \leqslant \sigma < \frac{1}{2},$ by the functional equation $\zeta(s) = \chi(s) \zeta(1-s)$ and the  asymptotic relation (see \cite[Page 95]{T}) that $\left|\chi(s)\right| \sim \left(\frac{t}{2\pi}\right)^{\frac{1}{2}-\sigma} $, as $t \to \infty$, we have \begin{align*}
     \left|\zeta(\sigma+it)\right| &\leqslant \left( 1+o(1) \right) \left(\frac{t}{2\pi}\right)^{\frac{1}{2}-\sigma}  \left|\zeta(1- \sigma +it)\right|\\ &\leqslant \exp\left\{\left(\frac{1}{2}-\sigma\right) \log \frac{t}{2\pi} + \frac{\log 2}{2}\,\frac{\log t }{\log \log t}  +O\left(\frac{\log t \, \log \log \log t }{(\log \log t)^2} \right)   \right\}\\ &\leqslant \exp\left\{\left(A+\frac{\log 2}{2}\right)   \,\frac{\log t }{\log \log t}  +O\left(\frac{\log t \, \log \log \log t }{(\log \log t)^2} \right)   \right\}\,, \quad \text{if}  \quad \frac{1}{2} -  \frac{A}{\log\log t} \leqslant \sigma < \frac{1}{2}.
 \end{align*}
 
Set $\delta = \frac{A}{\log\log t}$ and let $ A = \frac{1}{2}\epsilon$. By Cauchy's  formula  we obtain that
\begin{align*}
  \left|\zeta^{(\ell)}\Big(\frac{1}{2}+it\Big)\right| &\leqslant \frac{\ell!}{\delta^{\ell}} \max_{ \substack{ |\sigma -\frac{1}{2} | \leqslant \delta\\ |\widetilde{t} - t | \leqslant \delta}}\big|\zeta(\sigma+ i\widetilde{t})\big|\\ &\leqslant \frac{\ell!}{A^{\ell}} \left(\log\log t \right)^{\ell}  \exp\left\{\left(A+\frac{\log 2}{2}\right)   \,\frac{\log t }{\log \log t}  +O\left(\frac{\log t \, \log \log \log t }{(\log \log t)^2} \right)   \right\}\\ & \leqslant  \exp  \left\{\left(\frac{\log 2}{2}+\epsilon\right)\,\frac{\log t }{\log \log t}\right\}\,.
  \end{align*}

 \subsection{Proof of  Proposition \ref{strip} (B)}

 Let $A$ be a positive number, to be chosen later.  By   Lemma \ref{CC}, we have
 \begin{align*}
     \left|\zeta(\sigma+it)\right| \leqslant  \exp  \left\{ \left( \frac{1}{2} + \frac{2 \sigma -1}{\sigma(1-\sigma)}  \right) \, \frac{(\log t)^{2 - 2\, \sigma}}{\log \log t} + O\left(\frac{(\log t)^{2 - 2\sigma}}{(1-\sigma)^2(\log \log t)^2}\right)\right\}\,, \text{if}~    \left|\sigma - \sigma_0 \right| \leqslant \frac{A}{\log\log t}.
 \end{align*}
 
 Set $\delta = \frac{A}{\log\log t}$. Let $ A = \frac{1}{2} \log \left(1+ \widetilde{\epsilon}\, \right) $ with $ \widetilde{\epsilon}$ satisfying  $ \left( \frac{1}{2} + \frac{2 \sigma_0 -1}{\sigma_0(1-\sigma_0)}  \right)\left(1+\widetilde{\epsilon}\,\right) =  \frac{1}{2} + \frac{2 \sigma_0 -1}{\sigma_0(1-\sigma_0)} + \frac{1}{2}\epsilon $. By Cauchy's  formula  we obtain that
\begin{align*}
   \left|\zeta(\sigma_0+it)\right| &\leqslant \frac{\ell!}{\delta^{\ell}} \max_{ \substack{ |\sigma -\sigma_0 | \leqslant \delta\\ |\widetilde{t} - t | \leqslant \delta}}\big|\zeta(\sigma+ i\widetilde{t})\big|\\ &\leqslant \frac{\ell!}{A^{\ell}} \left(\log\log t \right)^{\ell}  \exp  \left\{ \left( \frac{1}{2} + \frac{2 \sigma_0 -1}{\sigma_0(1-\sigma_0)}  \right) \, \frac{(\log t)^{2 - 2\, \sigma_0} (\log t)^{2 \delta} }{\log \log t} + O_{\sigma_0}\left(\frac{(\log t)^{2 - 2\, \sigma_0} (\log t)^{2 \delta} }{(\log \log t)^2}\right)\right\}\\ & \leqslant  \exp  \left\{ \left( \frac{1}{2} + \frac{2 \sigma_0 -1}{\sigma_0(1-\sigma_0)}  \right) \, \frac{(\log t)^{2 - 2\, \sigma_0} (1+\widetilde{\epsilon}\,) }{\log \log t} + O_{\sigma_0, \,\ell, \, \epsilon}\left(\frac{(\log t)^{2 - 2\, \sigma_0} }{(\log \log t)^2}\right)\right\}\,\\ & \leqslant \exp  \left\{ \left( \frac{1}{2} + \frac{2 \sigma_0 -1}{\sigma_0(1-\sigma_0)} + \epsilon \right) \, \frac{(\log t)^{2 - 2\, \sigma_0}}{\log \log t} \right\}\,.
  \end{align*}

 \section{Proof of Proposition  \ref{logGCDstrip} }
 \subsection{Proof of Proposition  \ref{logGCDstrip} (\text{A})}\label{ProofCentral}
 We first prove the upper bound.  Let $\M \subset \N$ with $|\M| = N.$ Clearly, $m \leqslant [m, n]$ and $n \leqslant [m, n]$. By the inequality $\log^{\,2\ell}X \leqslant (200\,\ell)^{2\ell}\, X^{\frac{1}{100}},\,(\forall X \geqslant 1)$,  Rankin's trick  and \eqref{GCD: 1/2}, we have 
\begin{align*}
    &\sum_{m,n \in \M } \sqrt{ \frac{(m, n)}{[m , n]} }\log^{\ell} \Big(\frac{m}{(m,n)}\Big)\log^{\ell}\Big(\frac{n}{(m,n)}\Big)\\
    \leqslant & \sum_{m,n \in \M } \sqrt{ \frac{(m, n)}{[m , n]} }\log^{\,2\ell} \Big(\frac{[m , n]}{(m,n)}\Big)\\
    \leqslant & (200\,\ell)^{2\ell} \sum_{ \substack{ m,n \in \M \\ \frac{[m , n]}{(m,n)} \geqslant N^{10}  } } \sqrt{ \frac{(m, n)}{[m , n]} } \Big(\frac{[m , n]}{(m,n)}\Big)^{\frac{1}{100}} \,+ \sum_{ \substack{ m,n \in \M \\ \frac{[m , n]}{(m,n)} \leqslant N^{10}  } } \sqrt{ \frac{(m, n)}{[m , n]} }\log^{\,2\ell} \Big(\frac{[m , n]}{(m,n)}\Big)\\
    \leqslant & (200\,\ell)^{2\ell} \sum_{ \substack{ m,n \in \M   } }  \left(N^{10}\right)^{\frac{1}{100}-\frac{1}{2}} \,+ \log^{\,2\ell} \left(N^{10}\right) \sum_{ \substack{ m,n \in \M   } } \sqrt{ \frac{(m, n)}{[m , n]} }\\
    \leqslant & (200\,\ell)^{2\ell} N^{2+0.1-5} + \left(\log^{\,2\ell} \left(N^{10}\right)\right) \, N \,\exp \Big\{\big(2\sqrt{2} +o(1)\big)\sqrt{\frac{\log N \,\log_3 N}{\log_2 N}}\Big\}\\ \leqslant &N \,\exp \Big\{\big(2\sqrt{2} +o(1)\big)\sqrt{\frac{\log N \,\log_3 N}{\log_2 N}}\Big\}\,.
\end{align*}

Now we prove the lower bound.
 
 Let $\M$ be defined as in \cite[page 109]{delaBT},
 then $\M$ is a divisor-closed set with $|\M| \leqslant N $ and 
 \[ \frac{S_{\frac{1}{2}}(\M) }{|\M|} \geqslant \exp \Big\{\big(2\sqrt{2} +o(1)\big)\sqrt{\frac{\log N \,\log_3 N}{\log_2 N}}\Big\}\, .\]
 
 Since $\M$ is divisor-closed and $|\M| \leqslant N $, by \cite[Lemma 5.1]{delaBT},  we have
 \begin{align*}
     \widetilde{E}_{\frac{1}{2}}(\M):\,&= \sum_{\substack{m,n \in \M \\ m = (m, n) \,\text{or}\, n = (m, n) }} \sqrt{ \frac{(m, n)}{[m , n]} }\\ &\leqslant   2 \sum_{\substack{m,n \in \M \\ n = (m, n) }} \sqrt{ \frac{(m, n)}{[m , n]} }= 2 \sum_{\substack{m,n \in \M \\ n | m}} \sqrt{\frac{n}{m} } \leqslant |\M| \exp\left\{ \frac{2+o(1)}{\sqrt{2}} \sqrt{\frac{\log N}{\log_2 N}}   \right\}\,.
 \end{align*}
 
 We thus have
 \begin{align*}
     \sum_{\substack{m,n \in \M \\ \frac{m}{(m, n)}   \,, \frac{n}{(m, n)} \geqslant 2}} \sqrt{ \frac{(m, n)}{[m , n]} } \geqslant |\M|\, \exp \Big\{\big(2\sqrt{2} +o(1)\big)\sqrt{\frac{\log N \,\log_3 N}{\log_2 N}}\Big\}\,.
 \end{align*}
 
 Now we return to the log-type GCD sum and get
 \begin{align*}
 &\sum_{m,n \in \M } \sqrt{ \frac{(m, n)}{[m , n]} }\log^{\ell} \Big(\frac{m}{(m,n)}\Big)\log^{\ell}\Big(\frac{n}{(m,n)}\Big)\\
 \geqslant &\sum_{\substack{m,n \in \M \\ \frac{m}{(m, n)}   \,, \frac{n}{(m, n)} \geqslant 2}} \sqrt{ \frac{(m, n)}{[m , n]} }\log^{\ell} \Big(\frac{m}{(m,n)}\Big)\log^{\ell}\Big(\frac{n}{(m,n)}\Big)\\ \geqslant &(\log 2)^{2 \ell }\cdot |\M|\,\cdot \exp \Big\{\big(2\sqrt{2} +o(1)\big)\sqrt{\frac{\log N \,\log_3 N}{\log_2 N}}\Big\}\\\geqslant & |\M| \exp \Big\{\big(2\sqrt{2} +o(1)\big)\sqrt{\frac{\log |\M| \,\log_3 |\M|}{\log_2 |\M|}}\Big\}\,.
 \end{align*}

 \subsection{Proof of Proposition  \ref{logGCDstrip} (\text{B})}
 
The proof for the upper bound is almost the same as in Subsection \ref{ProofCentral}. The only difference is to use the upper bound in \eqref{gcd} instead of using  \eqref{GCD: 1/2}.

Now we consider the proof for the lower bound. Without loss of generality, assume that $N$ is a power of 2, i.e., $N = 2 ^k$ for $k\in \N$ (since any positive integer is between two powers of 2, it suffices to prove the statement for this case).

As in \cite[page 1526]{ABS}, let $\M$ be the set of   all square-free integers composed of the first $k$ primes (following ideas of \cite{G}), then $|\M| = 2^k = N$ and
 \begin{align}\label{GCD: sigma}
 S_{\sigma}(\M)= 2^k \prod_{i = 1}^k \left(  1+ \frac{1}{p_i^{\sigma}} \right) \geqslant N\,\cdot \emph{\emph{exp}}  \Big\{\frac{\widetilde{c}}{1-\sigma}\cdot  \frac{(\log N)^{1 - \sigma}}{(\log_2 N)^{\sigma}}\Big\}
\end{align}
 for some positive constant $\widetilde{c}$,~ by the prime number theorem.
 
 The multiplicative structure of $\M$ implies that
 \begin{align*}
     \widetilde{E}_{\sigma}(\M):\,= \sum_{\substack{m,n \in \M \\ m = (m, n) \,\text{or}\, n = (m, n) }} \frac{(m, n)^{\sigma}}{[m , n]^{\sigma}} \leqslant   2 \sum_{\substack{m,n \in \M \\ n = (m, n) }} \frac{(m, n)^{\sigma}}{[m , n]^{\sigma}} = 2 \sum_{\substack{m,n \in \M \\ n | m}} \left(\frac{n}{m} \right)^{\sigma} = 2^{k+1}\prod_{i = 1}^{k}\left(1+\frac{1}{2 p_i^{\sigma}}\right)\,.
 \end{align*}
 
 By the prime number theorem, we have
 \begin{align*}
    \frac{\widetilde{E}_{\sigma}(\M)}{S_{\sigma}(\M)} \leqslant 2 \prod_{i = 1}^{k} \frac{ \left(1+\frac{1}{2 p_i^{\sigma}}\right)}{\left(  1+ \frac{1}{p_i^{\sigma}} \right)} = \exp\left( -\frac{1}{2} \sum_{i = 1}^{k}\frac{1}{p_i^{\sigma}} + O\left(\sum_{i = 1}^{k}\frac{1}{p_i^{2\sigma}}  \right)\right) \longrightarrow 0, \quad \text{when} \quad k \to \infty.
 \end{align*}

From this, we obtain
 \begin{align*}
     \sum_{\substack{m,n \in \M \\ \frac{m}{(m, n)}   \,, \frac{n}{(m, n)} \geqslant 2}} \frac{(m, n)^{\sigma}}{[m , n]^{\sigma}} \geqslant \left(1+o(1)\right) \, N\, \,\emph{\emph{exp}}  \Big\{\frac{\widetilde{c}}{1-\sigma}\cdot  \frac{(\log N)^{1 - \sigma}}{(\log_2 N)^{\sigma}}\Big\}.
 \end{align*}
 
The remaining steps can be done as in  Subsection \ref{ProofCentral}.

 \section*{Acknowledgements}
 I  thank Christoph Aistleitner   and Kristian Seip for several helpful discussions on the  log-type GCD sums. I am grateful to  Andr\'{e}s Chirre
 for  many useful suggestions on conditional upper bounds for derivatives of the Riemann zeta function. I thank  Marc Munsch  for showing me a short proof of Lemma \ref{IntegralLimiBound}. Part of the work was done when I was visiting the Norwegian University of Science and Technology, and part of the paper was written when I was visiting Shandong University. I thank Kristian Seip and Yongxiao Lin for their hospitality.  The work was supported by the Austrian Science
Fund (FWF), project W1230.

%\subsection{transcendental}

\end{document}